\documentclass[11pt]{amsart}
\usepackage{hyperref}
\usepackage{amssymb, amsmath, amsthm, amsfonts}
\usepackage{verbatim}
\usepackage{bbm}
\usepackage{enumerate}
\usepackage{dsfont}
\usepackage[mathscr]{eucal}
\usepackage{tikz}

\usepackage{color}
\definecolor{gray}{gray}{0.5}

\def\pr{{\text{\rm pr}}} 
\def\sec{{\text{\rm sect}}} 
\def\R{\mathbb{R}}
\def\N{\mathbb{N}}

\def\Z{\mathbb{Z}}
\def\H{\mathcal{H}}
\def\beq{\begin{equation}}
\def\endeq{\end{equation}}

\def\mc{\mathcal}

\def\bs{\begin{split}}
	\def\es{\end{split}}

\def\tx{\tilde x}
\DeclareMathOperator{\re}{Re}
\DeclareMathOperator{\sgn}{sgn}

\def\cplus{{c\ci+}}
\def\cmin{{c\ci-}}

\def\mc{\mathcal}

\def\bbone{{\mathbbm 1}}

\newcommand{\sL}{\mathscr L}

\newcommand{\F}{\mathcal F}

\newcommand{\ci}[1]{_{{}_{\!\scriptstyle{#1}}}}

\newcommand{\Be}{\begin{equation}}
\newcommand{\Ee}{\end{equation}}

\newcommand{\Bm}{\begin{multline}}
\newcommand{\Em}{\end{multline}}
\newcommand{\Bea}{\begin{eqnarray}}
\newcommand{\Eea}{\end{eqnarray}}
\newcommand{\Beas}{\begin{eqnarray*}}
	\newcommand{\Eeas}{\end{eqnarray*}}
\newcommand{\Benu}{\begin{enumerate}}
	\newcommand{\Eenu}{\end{enumerate}}
\newcommand{\Bi}{\begin{itemize}}
	\newcommand{\Ei}{\end{itemize}}

\def\intslash{\rlap{\kern  .32em $\mspace {.5mu}\backslash$ }\int}
\def\qsl{{\rlap{\kern  .32em $\mspace {.5mu}\backslash$ }\int_{Q_x}}}

\def\N{\mathbb N}

\def\emph#1{{\it #1 }}

\def\Ga{\Gamma}
\def\ga{\gamma}

\def\cf{{\it cf}}

\def\sgn{{\text{\rm sign }}}

\def\loc{{\text{\rm loc}}}

\def\supp{{\text{\rm supp}}}

\def\inn#1#2{\langle#1,#2\rangle}

\def\noi{\noindent}

\def\meas{{\text{\rm meas}}}

\def\lc{\lesssim}
\def\gc{\gtrsim}

\def\eps{\varepsilon}
\def\ep{\epsilon}

\def\ka{\kappa}
            
\def\la{\lambda}

\def\om{\omega}              \def\Om{\Omega}

\def\fI{{\mathfrak {I}}}
\def\fJ{{\mathfrak {J}}}

\def\fN{{\mathfrak {N}}}

\def\fS{{\mathfrak {S}}}

\def\bbD{{\mathbb {D}}}
\def\bbE{{\mathbb {E}}}

\def\bbN{{\mathbb {N}}}

\def\bbR{{\mathbb {R}}}

\def\bbZ{{\mathbb {Z}}}

\def\cF{{\mathcal {F}}}

\def\cH{{\mathcal {H}}}

\def\cJ{{\mathcal {J}}}

\def\cL{{\mathcal {L}}}
\def\cM{{\mathcal {M}}}
\def\cN{{\mathcal {N}}}

\def\cP{{\mathcal {P}}}

\def\cR{{\mathcal {R}}}
\def\cS{{\mathcal {S}}}
\def\cT{{\mathcal {T}}}

\def\be#1{\begin{equation}\label{ #1}}
\def\endeq{\end{equation}}
\def\endal{\end{align}}
\def\bas{\begin{align*}}
\def\eas{\end{align*}}
\def\bi{\begin{itemize}}
\def\ei{\end{itemize}}

\def\eps{\varepsilon}

\def\emph#1{{\it #1}}
\def\textbf#1{{\bf #1}}

\def\beq{\begin{equation}}
\def\endeq{\end{equation}}

\def\bs{\begin{split}}
\def\es{\end{split}}

\theoremstyle{plain}
\newtheorem{thm}{Theorem}[section]
\newtheorem{prop}[thm]{Proposition}

\newtheorem{lem}[thm]{Lemma}
\newtheorem{cor}[thm]{Corollary}

\newtheorem*{thm*}{Theorem}
\newtheorem*{conj*}{Conjecture}
\newtheorem*{openproblem*}{Open Problem}

\numberwithin{equation}{section}

\begin{document}
\title
[Families  of Hilbert transforms along homogeneous curves]{A maximal function for  families  of Hilbert transforms along homogeneous curves}

\author[S. Guo \  \  \  \  \ \ J. Roos \ \ \ \ \ \ A. Seeger \ \  \ \  \ \ P.-L. Yung] {Shaoming Guo \ \ Joris Roos \ \ Andreas Seeger \ \ Po-Lam Yung}
\address{Shaoming Guo: Department of Mathematics, University of Wisconsin-Madison, 480 Lincoln Dr, Madison, WI-53706, USA}
\email{shaomingguo@math.wisc.edu}

\address{Joris Roos: Department of Mathematics, University of Wisconsin-Madison, 480 Lincoln Dr, Madison, WI-53706,USA}
\email{jroos@math.wisc.edu}

\address{Andreas Seeger: Department of Mathematics, University of Wisconsin-Madison, 480 Lincoln Dr, Madison, WI-53706, USA}
\email{seeger@math.wisc.edu}

\address{Po-Lam Yung: Department of Mathematics, The Chinese University of Hong Kong, Ma Liu Shui, Shatin, Hong Kong}
\email{plyung@math.cuhk.edu.hk}

\dedicatory{In memory of Eli Stein}

\begin{abstract}
	Let  $H^{(u)}$ be the Hilbert transform along the parabola $(t, ut^2)$ where $u\in \mathbb R$. For a set $U$ of positive  numbers consider the maximal function
	$\cH^U \!f= \sup\{|H^{(u)}\! f|: u\in U\}$. We obtain an 
	(essentially) optimal result for the $L^p$ operator norm of $\cH^U$ when $2<p<\infty$. The results are proved  for families of Hilbert transforms along more general nonflat homogeneous curves.
\end{abstract}

\date{\today}

\maketitle

\section{Introduction and statement of results}

Given $b>1$, $u>0$, consider the 
curve $$\Ga_{u,b}(t)=(t, u \gamma_b(t)), \quad   t\in \bbR,$$ where $\gamma_b$ is homogeneous of degree $b$, with $\gamma_b(\pm 1)\neq 0$. 

That is,  there are 
$c_+\neq 0$, $c_-\neq 0$ such that
\Be\label{gammabdef}\gamma_b(t)= \begin{cases}
	c\ci+ t^b,\,\,  &t>0,
	\\
	c\ci- (-t)^b, \,\,&t<0.
\end{cases}
\Ee

For $f\in \cS(\bbR^2)$  the  Hilbert transform along  $\Ga_{u,b}$ is defined by

\beq\notag
{H}^{(u)}\! f(x)= p.v. \int_{\R} f(x_1-t, x_2-u \gamma_b(t) )\frac{dt}{t}.
\endeq
For  an arbitrary  nonempty $U\subset \bbR$ consider the  maximal function
\beq\label{180722e1.2}
\mc{H}^U \!f(x)=\sup_{u\in U} |{H}^{(u)}\! f(x)|.
\endeq

The individual operators 
$H^{(u)}$  extend to  bounded operators on $L^p(\bbR^2)$ for $1<p<\infty$ (see 
\cite{SW78}, \cite{DR86}).  
The purpose of this paper is to prove, for $p>2$,  optimal  $L^p$ bounds for the maximal operator $\cH^U$ in terms of suitable  properties of  $U$.

Our maximal function is motivated 
by a similar one 
involving directional Hilbert transforms which correspond to the limiting case $b=1$, $c_+=-c_-$ not covered here.  This   maximal function  for Hilbert transforms along lines was considered by Karagulyan \cite{Kar07} who proved  that in this case the $L^2\to L^{2,\infty} $
operator norm is bounded below by $c \sqrt{\log(\# U)}$; the lower bound was extended to all $L^p$ by \L aba, Marinelli and Pramanik \cite{LMP17}.   Demeter and Di Plinio 
\cite{DD14} showed the  upper bound $O(\log (\#U))$ for $p>2$ (see also \cite{Dem10} for the sharp
$L^2$  result with bound $O(\log (\#U))$).  Moreover there is a sharp  bound 
$\approx \sqrt{\log (\#U)} $ for  lacunary sets of directions (see also Di Plinio and Parissis \cite{DP17}) and there are other improvements for   direction sets of Vargas type. 
Another  motivation for our work comes from the recent papers \cite{GHLR17}, \cite{DGTZ18}  which take up the curved cases and  analyze 
the linear operator $f\mapsto H^{(u(\cdot))}f$ for special classes  of measurable functions $x\mapsto u(x)$. 
\cite{GHLR17} covers the case when $u(x) $ depends only on $x_1$  
and \cite{DGTZ18} covers the case where $u$ is Lipschitz. The analogous questions for variable lines are still not completely resolved (\cf.  \cite{Bat13}, \cite{BT13} for partial $L^p$ ranges in the one-variable case, and
\cite{Guo15} and the references therein for partial results related to the Lipschitz case).

For our curved variant we seek to get sharp  results about the  dependence of the operator norm
\[
\|\cH^U\|_{L^p\to L^p} =\sup\{\|\cH^U\! f\|_p: \|f\|_p\le 1\}\] 
on $U$. 
Unlike  in the case for lines we obtain  for $b>1$ an optimal  bound when  $p>2$ and also observe a different type of dependence on $U$; namely it is not the cardinality of $U$ that determines  the size  of the operator norm for the maximal operator but rather  the minimal number of intervals of the form $(R,2R)$ that is needed to cover $U$. This number is  comparable to
\Be \label{fNUdef} \fN(U):=1 +\#\{n\in \bbZ: [2^n, 2^{n+1}] \cap U\neq \emptyset\}.\Ee

\begin{thm}\label{hil-max}
	For every $p\in (2,\infty)$, 
	the operator $\cH^U$ is bounded on $L^p$ if and only if $\fN(U)<\infty$. Moreover, 
	$$
	\|\cH^U\|_{L^p\to L^p}\approx 
	\sqrt{\log (\fN(U))}\,.
	$$
	The constants implicit in this equivalence depend only on $p$, $b$ and $|c_+/c_-|$.
\end{thm} 

\smallskip 

\noi{\it Remarks.} (i) The lower bound  $c\sqrt{\log(\fN(U))}$  can be extended to all $p>1$. Indeed, if we had a smaller operator norm for some $p_0<2$ we could, by interpolation, also  deduce a better upper bound for $p>2$ which is not possible. The lower  bound for $p<2$ is generally not efficient, see however some results for lacunary sets in \S\ref{lacunary-section}.

(ii)  
Concerning upper bounds there is no endpoint result for general $U$ with $\fN(U)<\infty$ when  $p=2$. In fact one can show  using the  Besicovitch  set that for $U=[1,2]$ the operator $\cH^U$ even fails to be of restricted weak type $(2,2)$. Cf. \cite[\S8.3]{STW03} for the details of a similar argument  in the context of  maximal functions for circular means.

(iii) In our theorem  we  avoid the cases $c_\pm =0$, for the following reasons.
For the case  $c_+=0=c_-$  in \eqref{gammabdef} the operators $H^{(u)}$ are equal to the Hilbert transform along a fixed line and the problems on $\cH^U$ become trivial. 
For the choices
$c_+\neq 0$, $c_-=0$ and 
$c_+=0$, $c_-\neq 0$ the  curves are  unbalanced and by
\cite[\S6]{CoR86} the individual operators $\cH^u$ are not bounded  on $L^p$.

(iv) 
The operators $\cH^{U}$ are invariant under conjugation with dilation operators with respect to the second variable; i.e. if $\delta^{(2)}_v\!f(x)=f(x_1, vx_2)$ then we have 
$\cH^{vU}=\delta_{v^{-1}}^{(2)}\cH^{U}\delta_v^{(2)}$ and thus the $L^p$ operator norm  of $\cH^{U}$ and $\cH^{vU}$ are the same.
This shows that any dependence of $c_+, c_-$ in the operator norms can always be reduced to a dependence on
just  $|c_+/c_-|$ as one can assume that $c_+=1$. The implicit constants in the above theorems  depend on $c_\pm, b,p$ but are uniform as long as  $|c_+/c_-| $ is taken in a compact subset of $(0,\infty)$,  and $b$ and $p$ are taken in  compact subsets of $(1,\infty)$.   Thus implicit constants in all inequalities in this paper will be allowed to depend on $c_\pm, b$, with the above understanding of boundedness on compact sets. 

\smallskip

\subsection*{\it This paper}
In \S\ref{deccomp-sect} we describe the basic decomposition \eqref{hilbtransformdec} of the Hilbert transform $H^{(u)}$ 
into a standard nonisotropic singular  integral operator $S^u$ and two operators $T^u_\pm$ which can be viewed as singular Fourier integral operators with favorable frequency localizations. The growth condition in terms of $\sqrt{\log \fN(U)}$ is only relevant for  the maximal function $\sup_{u\in U} |S^uf|$ for which we prove $L^p$ bounds for all $1<p<\infty$.  Here we use the Chang-Wilson-Wolff inequality, together with a variant of an approximation argument in \cite{GHLR17}.
It turns out that the full maximal operators associated to the $T^u_\pm$ are bounded in $L^p(\bbR^2)$ for $2<p<\infty$.
This is related to space-time $L^p$ inequalities  (so-called local smoothing estimates) for Fourier integral operators in \cite{MSS93}. This connection has already been used by Marletta and Ricci in their work \cite{MR} on families of maximal functions along homogeneous curves. The results for $S^u$, $T^u_\pm$ are formulated in \S\ref{deccomp-sect} as Theorems \ref{Suthm} and \ref{Tuthm}.

\S\ref{toolssect} contains several auxiliary results. A version of our maximal function for Mikhlin multipliers (dilated in the second variable)  is given in \S\ref{mikhlinmaxsect}; 
this is used to prove Theorem  \ref{Suthm} in \S\ref{Suproofsect}.  Theorem \ref{Tuthm} is proved in \S\ref{Tuproofsect}.
In \S\ref{lacunary-section} we prove some results 
about upper bounds for the maximal functions
$\sup_{u\in U} |T^u_\pm f|$ when $U$ is a lacunary set; one of these  results will be helpful in the proof of lower bounds for the operator norm.

The  proof of lower bounds  is given in \S\ref{karagulyan-sect}. 
The arguments for the  lower bounds in $L^2$  are based on ideas 
of Karagulyan \cite{Kar07}. Appendix \ref{cotlar-appendix}
contains a  Cotlar type  inequality which is used in the proof of Theorem \ref{Suthm}. 

\subsection*{\it Acknowledgements} S.G. was supported in part 
	by a direct grant for research from the Chinese University of Hong Kong (4053295).
	A.S. was supported in part by National Science Foundation grants DMS 1500162 and 1764295.
	He would like to thank the Isaac Newton Institute for Mathematical Sciences, Cambridge, for support and hospitality during the program 
 Approximation, Sampling and Compression in Data Science where some work on this paper was undertaken. This work was supported by EPSRC grant no EP/K032208/1.
	P.Y. was supported in part by a General Research Fund CUHK14303817 from the Hong Kong Research Grant Council, and direct grants for research from the Chinese University of Hong Kong (4441563 and 4441651).\\
The authors thank Rajula Srivastava for reading a draft of this paper and for providing useful comments.

\section{Decomposition of the Hilbert transforms}\label{deccomp-sect}

Let $\chi_+$ be supported in $(1/2,2)$ such that $\sum_{j\in \bbZ} \chi_+(2^j t) =1$ for $t>0$.
Let $\chi\ci-(t)=\chi_+(-t)$ and $\chi=\chi\ci+ + \chi\ci-$.
We define measures $\sigma\ci+$ and $\sigma\ci-$ by
\Be\label{sigmapmdef}
\inn{\sigma\ci\pm} {f} = \int f(t,\ga_b(t)) \chi\ci{\pm}(t) \frac {dt} {t}.
\Ee
Let, for $j\in \bbZ$,   the measure $\sigma_j$ be defined by 
\[\inn{\sigma_j  } {f} = \int f(t,\ga_b(t)) \chi(2^j t) \frac {dt} {t}.\]
By homogeneity of $\gamma_b$ we see that (in the sense of distributions)
$\sigma_j = 2^{j(1+b)} \sigma_0 (\delta_{2^j}^b \cdot)$ with  $\delta_{t}^bx= (tx_1, t^bx_2)$.
Observe that  $\sigma_0=\sigma\ci+ +\sigma\ci-$ satisfies the cancellation condition $\widehat \sigma_0(0)=0$ (where $\widehat \sigma (\xi)\equiv \cF[\sigma] (\xi) = \int e^{-i\inn x\xi} d\sigma(x)$ denotes the Fourier transform). 
For Schwartz functions $f$ the Hilbert transform along $\Gamma_b$ is then given by
$$H f= \sum_{j\in \bbZ}{\sigma_j}* f.$$

\subsection{\it Asymptotics for the  Fourier transform of $\sigma_0$}
We analyze  $\widehat {\sigma_\pm}(\xi)$ for large $\xi$. 
We have
\[
\widehat {\sigma\ci\pm} (\xi) = 
\int e^{-i\psi_\pm(t, \xi)}  \chi\ci\pm ( t) \frac {dt} {t}\]
with
\begin{align*} \psi\ci+(t,\xi) &= t \xi_1  + \cplus t^b\xi_2 ,
\\
\psi\ci-(t,\xi) &= t \xi_1  + \cmin (-t)^b\xi_2 .
\end{align*} 
Observe that \Be \label{firstder}
\begin{aligned}\partial_t\psi\ci+(t,\xi) &= \xi_1+ c\ci+ b t^{b-1}\xi_2, \quad 
	\\ \partial_t\psi\ci-(t,\xi) &= \xi_1- \cmin b (-t)^{b-1}\xi_2.
\end{aligned}\Ee
Thus $\psi\ci+$ has a critical point $t_+(\xi)>0 $ when $\xi_1/(c\ci+\xi_2)<0$,
and $\psi\ci-$ has a critical point $t_-(\xi) <0$ when $\xi_1/(c\ci-\xi_2)>0$ ,
and $t_\pm(\xi)$ are given by
$$t_+(\xi)= \Big( \frac{-\xi_1}{bc_+\xi_2}\Big)^{\frac 1{b-1}}, 
\qquad
t_-(\xi)= - \Big( \frac{\xi_1}{bc_-\xi_2}\Big)^{\frac 1{b-1}}.
$$
These critical points are nondegenerate as we have
$$\partial_{tt}\psi\ci\pm(t,\xi)=c\ci\pm b(b-1) (\pm t)^{b-2}\xi_2.$$
Setting
$\Psi_\pm(\xi)=-\psi\ci\pm(t_\pm(\xi),\xi)$ we get
\begin{align*}
\Psi\ci+(\xi)
&=  (b-1) c\ci+\xi_2\Big( -\frac{\xi_1}{bc\ci+\xi_2}\Big)^{\frac {b}{b-1}},
\\
\Psi\ci-(\xi) 
&= (b-1) c\ci-\xi_2\Big( \frac{\xi_1}{bc\ci-\xi_2}\Big)^{\frac{b}{b-1}}.
\end{align*}
The functions $\Psi_\pm$ are homogeneous of degree one and putting $\xi_2=\pm 1$ we have 
the crucial lower bounds for the 
second derivatives of  $\xi_1\mapsto \Psi(\xi_1,\pm 1)$ needed for the application of the space time estimate in \S\ref{lstheoremsect}.

Assume $|\xi|>1$. We observe  that then 
\begin{subequations}
	\Be \label{lowerbdpsiplus} \inf_{1/3 \le t\le 3} \big|\partial_t \psi\ci+ (t,\xi)\big| \gc |\xi|
	\Ee if $\xi_1/c_+\xi_2$ does not belong to the interval $[-b (7/2)^{b-1} , -b (2/7)^{b-1} ]$.
	
	Likewise, again for $|\xi|>1$ we  observe that 
	\Be\label{lowerbdpsimin}
	\inf_{-3\le t\le -1/3} \big|\partial_t \psi\ci-(t,\xi) \big| \gc |\xi|\Ee 
\end{subequations}  if $\xi_1/c_-\xi_2$ does not belong to the interval $[b (2/7)^{b-1} , b (7/2)^{b-1} ]$.
These observations  suggest the following decomposition of $ \sigma_0$.

Let $\eta_0$ be supported in $\{|\xi|\le 100\}$  and equal to $1$ for $|\xi|\le 50$.
Let $\varsigma_+$ be a $C^\infty_c(\bbR)$ function supported on
$(b (1/4)^{b-1} , b 4^{b-1} )$ which is equal to $1$ on 
$[b (2/7)^{b-1} , b (7/2)^{b-1} ]$.
Let $\varsigma_-$ be a $C^\infty_c(\bbR)$ function supported on
$(-b 4^{b-1} , -b (1/4)^{b-1} )$ which is equal to $1$ on 
$[-b (7/2)^{b-1} , -b (2/7)^{b-1} ]$.
Then we decompose
\begin{subequations}
	\Be \label{sigmadecomp} 
	\sigma_0= \phi_0+\mu_{0,+} +\mu_{0,-}
	\Ee
	where $\phi_0$ is given by
	\Be\begin{aligned} \label{defofPhi}
		\widehat{ \phi_0}(\xi) =
		\eta_0(\xi) \widehat \sigma_0(\xi) 
		&+ (1-\eta_0(\xi)) \big(1- \varsigma_- (\tfrac{\xi_1}{c_+\xi_2})\big)\widehat \sigma_+(\xi) 
		\\
		&+ (1-\eta_0(\xi)) \big(1- \varsigma_+(\tfrac{\xi_1}{c_-\xi_2})\big)\widehat \sigma_-(\xi) 
	\end{aligned}\Ee
	and $\mu\ci{0,\pm}$ are given by
	\begin{align}
	\widehat \mu_{0,+} (\xi) &= (1-\eta_0(\xi))  \varsigma_- (\tfrac{\xi_1}{c_+\xi_2})\widehat \sigma_+(\xi) ,
	\\
	\widehat \mu_{0,-} (\xi) &= (1-\eta_0(\xi))  \varsigma_+(\tfrac{\xi_1}{c_-\xi_2})\widehat \sigma_-(\xi) .
	\end{align}
\end{subequations}

\begin{lem}\label{symbollemma} 
	(i) $\phi_0$ is a Schwartz function with
	$\widehat \phi_0(0)=0$. 
	
	(ii) The function  $\widehat \mu_{0,+}$  is supported on
	\begin{subequations}
		\Be\label{sectorforomegaplus}
		{\rm Sect}\ci+= \big\{\xi: |\xi|>50, \,\,
		-b 4^{b-1} < \frac{\xi_1}{c_+\xi_2} < -\frac{b}{4^{b-1} }\big\}
		\Ee 
		and   
		satisfies 
		\[\widehat \mu_{0,+} (\xi) 
		= \omega_+(\xi) e^{i\Psi_+(\xi) } + E_+(\xi)
		\]
		where $\omega_+$ is a standard symbol of order $-1/2$, and $E_+(\xi)$ is a Schwartz function, both supported on $\rm{Sect}\ci+$.
		
		(iii) The function  $\widehat \mu_{0,-}$  is supported on
		\Be \label{sectorforomegaminus}
		{\rm Sect}\ci-= \big\{\xi: |\xi|>50, \,\,
		\frac{b}{4^{b-1}} < \frac{\xi_1}{c_-\xi_2} < b 4^{b-1} \big\}
		\Ee 
	\end{subequations}
	and   
	satisfies 
	$$\widehat \mu_{0,-} (\xi) = \omega_-(\xi) e^{i\Psi_-(\xi) } + E_-(\xi)$$
	where $\omega_-$ is a standard symbol of order $-1/2$, and $E_-(\xi)$ is a Schwartz function,
	both supported on $\rm{Sect}\ci-$.
\end{lem} 
\begin{proof} 
	In view of the lower bounds for $\partial_t\psi_\pm$ stated in 
	\eqref {lowerbdpsiplus}, \eqref {lowerbdpsimin} under their respective assumptions 
	we see that  $\phi_0$ is a Schwartz function. We have that $\widehat \sigma_+(0)=-\widehat \sigma_-(0)$ and it follows that $\widehat \phi_0(0)=0$.
	The formulas for 
	$\widehat \mu_{0,\pm} (\xi) $  follow by the method of stationary phase.
\end{proof}

We now define  $\Phi_0$ by $\widehat \Phi_0=\widehat \phi_0+ E\ci+ +E\ci-$  so that $\Phi_0$ is a Schwartz function with $\widehat \Phi_0(0)=0$.
Define $\Phi_j$ , $\ka_{j,\pm} $
by $$\widehat \Phi_j(\xi)=\widehat \Phi_0 (2^{-j}\xi_1, 2^{-jb}\xi_2)$$ and
$$\widehat {\kappa_{j,\pm}}(\xi) = \omega_\pm (2^{-j}\xi_1, 2^{-jb}\xi_2)
e^{i \Psi_\pm(2^{-j}\xi_1, 2^{-jb}\xi_2)} .
$$ Define  operators $S^u$ and $T^u_\pm$ by 
\begin{align} 
\widehat {S^u f}(\xi) &= \sum_{j\in \bbZ} \widehat \Phi_j (\xi_1, u\xi_2)\widehat f(\xi) 
\\
\widehat {T_\pm^u f}(\xi) &= \sum_{j\in \bbZ} \widehat {\kappa_{j,\pm}}  (\xi_1, u\xi_2)\widehat f(\xi) 
\end{align}

These expressions are at least well defined if $f$ is a Schwartz function whose Fourier transform is compactly supported in $\bbR^2\setminus \{0\}$.
For these functions we have then decomposed our Hilbert transform as
\Be \label{hilbtransformdec}H^{(u)}\!f = S^u f + T_+^u f + T^u_-f .\Ee

\medskip 

For the upper bound in Theorem \ref{hil-max} we shall prove 

\begin{thm} \label{Suthm} For $1<p<\infty$,
	\Be
	\big\|\sup_{u\in U} |S^u f| \big\|_p \lc \sqrt{\log(\fN(U))} \|f\|_p.
	\Ee
\end{thm}
\begin{thm} \label{Tuthm}For $2<p<\infty$,
	\Be \big\|\sup_{u>0} |T_\pm^u f| \big\|_p
	 \lc  \|f\|_p.
	\Ee
\end{thm}

\section{Auxiliary results}\label{toolssect}

\subsection{\it  The Chang-Wilson-Wolff inequality}

We consider the conditional expectation operators $\bbE_j$ generated by dyadic cubes  of length $2^{-j}$, i.e. intervals of the form $\prod_{i=1}^d[n_i2^{-j}, (n_i+1)2^{-j})$ with $n\in \bbZ^d$.
Let $f\in L^1_\loc(\bbR^d)$.
For each $j \in \N \cup \{0\}$, $\bbE_j$ is given by
$$
\bbE_j f(x) = \frac{1}{2^{-jd}} \int_{I_j(x)} f(y) dy
$$
where $I_j(x)$ is the unique dyadic cube  of side length $2^{-j}$
that contains $x$. Let $$\bbD_j = \bbE_{j+1} - \bbE_j$$ be the martingale difference operator. Let $\fS f$ be the dyadic square function, defined by
$$
\fS f(x) = \Big( \sum_{j\in\Z} |\bbD_j f(x)|^2 \Big)^{1/2}.
$$
Also let $\mathcal{M}$ be the dyadic maximal function, given by $$\mathcal{M} f(x) = \sup_{j\in\Z} |\bbE_j f(x)|.$$
The following is a slight variant of an inequality due to Chang, Wilson and Wolff \cite{CWW85}:
\begin{prop} \label{prop:CWW85} 
	Suppose 
	that
	$f\in L^p(\bbR^d)\cap L^\infty(\bbR^d)$ for some $p<\infty$.
	Then
	there exist two universal constants $c_1$ and $c_2$ 
	such that 
	\Be\label{CWWineq}\begin{aligned} \meas\Big(
		&\Big \{x \in \R^d \colon |f(x)| > 4 \lambda \text{ and } \fS f(x) \leq \varepsilon \lambda \Big \} \Big)
		\\
		&\qquad
		\leq c_2 \exp(- c_1\eps^{-2})  \meas\Big(\Big\{x \in \R^d \colon \mathcal{M} f(x) > \lambda \Big \} \Big)
	\end{aligned}\Ee
	for all $\lambda > 0$ and $0 < \varepsilon < 1/2$.
\end{prop}
This is a scaling invariant version of the Chang-Wilson-Wolff inequality. For a detailed proof see Appendix B.

We shall apply the one-dimensional version of this theorem for the vertical slices in $\bbR^2$.
Let $f$ be a measurable function in $L^p(\bbR^2)\cap L^\infty(\bbR^2)$, and for $j \geq 0$, let $\bbE_j^{(2)}$ be the  conditional expectation operator  acting on the second variable, i.e.
$$
\bbE_j^{(2)} \!f(x) = \frac{1}{2^{-j}} \int_{I_j(x_2)} f(x_1,y) dy
$$
where $I_j(x_2)$ is the unique dyadic interval of length $2^{-j}$ that contains $x_2$. Let $\bbD_j^{(2)} = \bbE_{j+1}^{(2)} - \bbE_j^{(2)}$, and 
$$
\fS^{(2)}\! f(x) = \Big( \sum_{j\in\Z} |\bbD_j^{(2)} f(x)|^2 \Big)^{1/2}.
$$
Then from the above proposition, we clearly have
\beq \label{eq:CWW85}
\begin{split}
	&\meas\Big( \Big\{x \in \R^2 \colon |f(x)| > 4 \lambda \text{ and } \fS^{(2)} \!f(x) \leq \varepsilon \lambda \Big\} \Big)
	\\
	&\qquad\qquad\qquad\leq  \, c_2 e^{-c_1\eps^{-2}} \meas\Big(\Big \{x \in \R^2 \colon \mathcal{M}^{(2)} \!f(x) > \lambda \Big \} \Big)
\end{split}
\endeq
for all $\lambda > 0$ and $0 < \varepsilon < \tfrac12$, where $\mathcal{M}^{(2)}$ is the dyadic maximal function in the second variable, i.e. $\mathcal{M}^{(2)}\! f(x) = \sup_{j\in\Z} |\bbE^{(2)}_j \!f (x)|.$

\subsection{\it Martingale difference operators and Littlewood-Paley projections}
We need some  computations from \cite{GHS06} which are  summarized in the following lemma.
Let $M$ denote the Hardy-Littlewood maximal operator acting on functions in $L^p(\bbR)$.
Let $\phi$ be supported in $( c^{-1},c)\cup (-c,-c^{-1})$ for some $c>1$.

\begin{lem}\label{gss-sublemma} Assume that $f\in L^1+L^\infty(\bbR)$.
	Then
	
	(i) For $q\ge 1$, $n\ge 0$,
	$$\bbE_k (\cF^{-1}[\phi(2^{-k-n}\cdot) \widehat f])(x) \lc 2^{-n(1-\frac 1q) } \big(M(|f|(x)^q)\Big)^{1/q}$$
	
	(ii)  For $n\ge 0$
	$$\bbD_k (\cF^{-1} [\phi(2^{-k+n}\cdot) \widehat f])(x) \lc 2^{-n } Mf(x)$$
	almost everywhere.
\end{lem}
\begin{proof}[Proof of Lemma \ref{gss-sublemma}] Cf. Sublemma 4.2 in \cite{GHS06}. \end{proof}
Given a function on $\bbR^2$ we shall apply this lemma to $y_2\mapsto f(y_1,y_2)$ and relate the 
square function $\fS^{(2)}$ to Littlewood-Paley square functions in the second variable.

Let $\chi_b$ be an even $C^\infty$ function supported in $(2^{-b}, 2^b)\cup(-2^b, -2^{-b})$  such that
$\sum_{k\in \bbZ} \chi_b(2^{-kb} t)=1$ for all $t\neq 0$.
Define the Littlewood-Paley projection type operators $P^{(1)}_k$,
$P^{(2)}_{k,b}$ acting on Schwartz functions on $\bbR^2$ by
\begin{align}\label{projectionops}
\widehat{P^{(1)}_k\! f} (\xi)&= \chi_1(2^{-k}\xi_1) \widehat f(\xi)
\\ \label{projectionopsb}
\widehat{P^{(2)}_{k,b} f} (\xi)&= \chi_b(2^{-kb}\xi_2) \widehat f(\xi)
\end{align}

\begin{lem}  \label{martvsLP-lem}
	Let $q>1$, $b>0$,  and let $g\in L^1+L^\infty$. Then  the pointwise inequality 
	$$\fS^{(2)}\! g \le C_{b,q} 
	\Big(\sum_{k\in \bbZ} \big[
	M^{(2)}( |P^{(2)}_{k,b} g|^q) ]^{2/q} \Big)^{1/2}
	$$
	holds
	almost everywhere. Here $M^{(2)}$ denotes the Hardy-Littlewood maximal operator in the second variable. 
\end{lem}
\begin{proof}[Proof of Lemma \ref{martvsLP-lem}.] Let $\phi_b$ be a $C^\infty$ function with 
\[ \supp(\phi_b) \subset (2^{-b}, 2^b)\cup(-2^b, -2^{-b})\]  which equals $1$ on the support of $\chi_b$. Define 
	$\widehat{\tilde P^{(2)}_{k,b} f} 
	(\xi)= \phi_b(2^{-kb}\xi_2) \widehat f(\xi).$
	We write $$\bbD_k^{(2)}= \sum_{n\in \bbZ} 
	\sum_{\substack{l\in \bbZ:\\n\le k-lb<n+1}} 
	\bbD_k^{(2)} \tilde P^{(2)}_{l,b}  P^{(2)}_{l,b} $$ and use Minkowski's inequality and Lemma 
	\ref{gss-sublemma} to estimate, with  $\eps<1-1/q$, 
	\begin{align*}\fS^{(2)} f &\lc  \sum_{n\in \bbZ} 2^{-|n|\eps} \Big(\sum_{k=0}^\infty \Big[
	\sum_{\substack{l\in \bbZ:\\n\le k-lb<n+1}} 
	M^{(2)}(|P^{(2)}_{l,b} f|^q)\Big]^{2/q}\Big)^{1/2}
	\\&\lc
	\Big(\sum_{l\in\bbZ}\big [
	M^{(2)}(|P^{(2)}_{l,b} f\big| ^q)|^{2/q}\Big)^{1/2}. \qedhere
	\end{align*}
	This finishes the proof of Lemma \ref{martvsLP-lem}.
\end{proof}
\subsection{\it A variant of Cotlar's inequality} \label{cotlarsect}
Recall that  $\chi_+\in C^\infty_c(\bbR)$ be supported in  $(1/2,2)$ such that
$\sum_{j=-\infty}^\infty \chi_+(2^j t)=1$ for $t>0$ and let $\eta= \chi_+(|\cdot|)$.

Consider a 
Mikhlin-H\"ormander multiplier $m$ on $\bbR^d$ satisfying the assumption 
\Be\label{HMcond}\sup_{t>0} \|\eta \,m(t\cdot)\|_{\sL^1_{\alpha}} =:B(m)<\infty, \quad \alpha>d;
\Ee
here $\sL^1_{\alpha}$ is the potential space of functions
$g$  with $(I-\Delta)^{\frac{\alpha}{2}}g\in L^1$.
Let $Sf=\cF^{-1}[m\widehat f]$, and for 
$n\in \bbZ$ let
$S_n$ be defined by 
\[\widehat {S_n f}(\xi) = \sum_{j\le n} \eta(2^{-j}\xi) m(\xi)\widehat f(\xi).
\]
Then both $S$ and the $S_n$ are of weak type $(1,1)$ 
and bounded on $L^p$ for $p \in (1,\infty)$ with uniform operator norms $\lc_p B(m)$. We are interested in bounds for the  maximal function 
\Be \label{S*def} S_* f(x)= \sup_{n\in \bbZ} |S_nf(x)|\Ee

\begin{prop}\label{Cotlarprop} Let $\alpha>d$,  $r>0$ and $B(m)$ as in \eqref{HMcond}. For $f\in L^p(\bbR^d)$,
	we have, for almost every $x$, and for $0<\delta\le 1/2$
	\Be\label{Sstarineq}
	S_* f(x) \le 
	\frac{1}{(1-\delta)^{1/r}}\big(M (|Sf|^r)(x)\big)^{1/r} + C_{d,\alpha} \delta^{-1} B(m)
	M f(x).
	\Ee
\end{prop}
Proposition \ref{Cotlarprop} is a  variant of the standard Cotlar inequality regarding truncations of 
singular integrals. A proof is included in  Appendix \ref{cotlar-appendix}.

\subsection{\it An $L^p$ space time  estimate for Fourier integral operators of convolution type and vector valued extensions}
\label{lstheoremsect}
Let $S(a_0, a_1)$ be the sectorial region  in $\bbR^2$ 
$$S(a_0,a_1)=\{(\xi_1,\xi_2):  a_0<|\xi_1|/|\xi_2|< a_1 , \,\xi_2>0\}$$
and let $\eta_\sec$ be $C^\infty$ and compactly supported in $S_{\text{ann}}:=
S (a_0,a_1)\cap \{\xi: 1<|\xi|<2\}$.
Let $q\in C^\infty$ be defined in $S(a_0,a_1)$ and  homogeneous of degree one, satisfying
$$q_{\xi\xi}\neq 0 \text{ on } S(a_0,a_1)$$
i.e. the Hessian $q_{\xi\xi}$ has rank one on the sector $S(a_0, a_1)$.
Model cases for $q(\xi)$ are given by $|\xi|$, or $\xi_1^2/\xi_2$ in the sector $\{|\xi_1|\le c|\xi_2|\}$.
Define
$$
F_R f(x,t) = \int e^{i (\inn x\xi+t q(\xi))} \eta_\sec(\xi/R) \widehat f(\xi) d\xi.
$$
We need a so-called local smoothing estimate 
from \cite{MSS93} (the terminology is supposed to indicate that the integration over a compact time interval improves on the fixed time estimate $\|F_R f(\cdot,t)\|_p \lc 
R^{\frac 12-\frac 1p} \|f\|_p$, $2\le p<\infty$).
\begin{thm*}\cite{MSS93}
	If $I$ is a compact interval then
	\Be\label{lsineq} 
	\Big(\int_I \int_{\bbR^2} |F_R f(x,t) |^p dx\, dt\Big)^{1/p} \lc C_I R^{\frac 12-\frac 1p-\eps(p) } \|f\|_p,
	\Ee
	with $\eps(p)>0$ if $2<p<\infty$.
	The estimates are uniform as  $\eta_\sec$ ranges over a bounded subset of $C^\infty$ functions supported in $S_{\text{\rm ann}}$.
\end{thm*}

In this paper we shall need a square-function  extension of \eqref {lsineq} which involves  nonisotropic dilations of the associated multipliers 
of the form $\xi\mapsto (2^{-j}\xi_1, 2^{-bj} \xi_2)$ with $b\ge 1$, $j\in \bbZ$
(the strict inequality  $b>1$ assumed in the introduction is not used here); see (\ref{eq:vvls}) below.
We rely on  a  variant of a theorem in 
\cite{See88}, for families of smooth multipliers $\xi\mapsto m(\xi,t)$  on $\bbR^d$ depending continuously on the parameter 
$t\in I$, where $I$ is a compact interval. Let $\cP$ be a real matrix whose eigenvalues have positive real parts and consider the dilations $\delta_s=\exp(s\log \cP)$.
\begin{prop} \label{globalmultthm} 
	Let $2<p<\infty$ and $I \subset \bbR$ be a compact interval. Recall that $\eta$ is a radial non-trivial $C^\infty$ function with support in $\{\xi:1/2<|\xi|<2\}$. Suppose $$\sup_{t\in I} \sup_\xi |m(\xi,t)|\le A,$$ and assume that for all $f\in \cS(\bbR^d)$, 
	\[\sup_{s>0}\Big(\frac{1}{|I|}\int_I\big\| \cF^{-1}[\eta m(\delta_s \cdot, t)  \widehat f ] \big\|_p^pdt \Big)^{1/p} \le A\|f\|_p.
	\]
	Moreover, suppose that for all multiindices $\alpha$ with $|\alpha_1|+|\alpha_2|\le d+1$,
	$$\big|\partial_\xi^{\alpha} [\eta(\xi) m(\delta_s \xi,t) ]\big| \le B, \quad t\in I, s>0.$$ Then there is a constant $C_p>0$ such that 
	\Be
	\Big(\frac{1}{|I|}\int_I
	\big\| \cF^{-1}[m(\cdot, t)  \widehat f ] \big\|_p^pdt \Big)^{1/p} \le C_p A \log (2+ B/A)^{1/2-1/p} \|f\|_p.
	\Ee
\end{prop} 

The proof is exactly the same as the proof for standard multipliers in \cite{See88}.
We shall use   the following consequence for a square function  inequality to derive (\ref{eq:vvls}).

\begin{cor} \label{vectmultthm}  Let $2<p<\infty$ and $I\subset \bbR$ be a compact interval. 
	Suppose that  there is a compact subset $K\subset \bbR^2\setminus\{0\}$
	such that $m_0(\xi,t)=0$ if $\xi\in K^\complement$ or $t\in I^\complement$. 
	Suppose that for all multiindices $\alpha$ with $|\alpha_1|+|\alpha_2|\le 10$,
	$$|\partial_\xi^{\alpha} m_0(\xi,t) | \le B, \quad t\in I,$$ and that 
	$$\sup_{t\in I} \sup_{\xi}|m_0(\xi,t)|\le A.$$ 
	Moreover, suppose that for all $f\in \cS(\bbR^2) $ the inequality
	$$\Big(\frac{1}{|I|} \int_I\big\| \cF^{-1}[m_0(\cdot,t) \widehat f ] \big\|_p^pdt \Big)^{1/p} \le A\|f\|_p$$
	holds. Define $T_j f(x,t) $ by $\widehat {T_jf}(\xi,t) = m_0(\delta_{2^{-j}} \xi,t)  \widehat f(\xi)$.
	Then there is a constant $C(K,p)$ such that for all $\{f_j\} \in L^p(\ell^2)$ we also have
	\begin{multline} \label{vectineq-m} 
	\Big(\frac{1}{|I|} \int_I \Big\|\Big(\sum_{j\in \bbZ} |T_jf_j(\cdot,t)|^2\Big)^{1/2} \Big\|_p^p dt \Big)^{1/p} \\ \le C(K,p)
	A \log (2+ B/A)^{1/2-1/p} 
	\Big\|\Big(\sum_j|f_j|^2\Big)^{1/2}\Big \|_p.
	\end{multline}
\end{cor}

\begin{proof}[Proof of Corollary \ref{vectmultthm}.] This is a straightforward consequence of Proposition~\ref{globalmultthm} (alternatively one can adapt the proof of Proposition~\ref{globalmultthm} to a vector-valued setting). Let $\widetilde{\phi} \in C^\infty_c(\bbR^d\setminus\{0\})$ such that $\widetilde{\phi}(\xi)=1$ for $\xi\in K$. Let $\cJ$
	be a subset of integers with the property that the supports of 
	$\widetilde{\phi}(\delta_{2^{-j}}\cdot)
	$, $j\in  \cJ$ are disjoint. 
	We may write  $\bbZ$ as union over $C_K$ such families.
	It is sufficient to show the analogue of \eqref{vectineq-m}  with the $j$-summation extended over $\cJ$.
	It will be convenient to work with an 
	enumeration  $\{j_1,j_2,\dots \}$  
	of $\cJ$.

	Let $L_j$ be defined by $\widehat {L_j f} =
	\widetilde{\phi}(\delta_{2^{-j}}\xi)\widehat f(\xi).$
	Let $g= \sum_i L_{j_i} f_{j_i} $; then
	by the adjoint version of the Littlewood-Paley inequality
	we have
	\Be\label{dualLP}
	\|g\|_p\lc 
	\Big\|\Big(\sum_i |f_{j_i}|^2\Big)^{1/2} \Big\|_p.
	\Ee
	Notice that \Be \label{Tjg} T_j g= T_j f_j\Ee by the disjointness condition  on the supports of 
	$\phi(\delta_{2^{-j_i}}\cdot)$.
	Let  $\{r_i\}_{i=1}^\infty$ 
	denote the sequence of Rademacher functions. 
	Applying Proposition~\ref{globalmultthm} to the multipliers $$m_\alpha(\xi)=\sum_{i=1}^\infty r_i(\alpha) 
	m_0(\delta_{2^{-j_i}}\xi,t) $$ and the function $g= \sum_{i=1}^\infty
	\cF[ \widetilde{\phi}(\delta_{2^{-j_i}}\cdot) \widehat f_{j_i}]$
	we get 
	\Be
	\Big(\int_0^1 \frac{1}{|I|} \int_I\big\| \cF^{-1}[m_\alpha(\cdot, t)  \widehat g ] \big\|_p^pdt\,d\alpha\Big)^{1/p} \lc A \log (2+ B/A)^{1/2-1/p} \|g\|_p\,.
	\Ee

	By interchanging the $\alpha$-integral  and the $(x,t)$-integral and applying Khintchine's inequality we obtain
	\[
	\Big(\frac{1}{|I|}  \int_I \Big\|\Big(\sum_{j\in \bbZ} |T_jg(\cdot,t)|^2\Big)^{1/2} \Big\|_p^p dt \Big)^{1/p} \lc 
	A \log (2+ B/A)^{1/2-1/p} \|g\|_p
	\] and the proof is completed  by applying
	\eqref{dualLP} and  \eqref{Tjg}. 
\end{proof}

\subsection{{\it A version of the Marcinkiewicz multiplier theorem}}\label{marcsect}
In the proof of Proposition \ref{lacthm} we shall use a well known version of the Marcinkiewicz multiplier theorem with minimal assumptions on the number of derivatives. Let 
$\eta_{\pr}$ be a nontrivial $C^\infty_c$ function which is even in all variables and supported in 
$\{\xi: 1/2<|\xi_i|\le 2, i=1,2\}$. Let $\sL^{
	2}_{\alpha,\alpha}$ 
the Sobolev space with mixed dominating smoothness consisting of $g\in L^2$ such that $$\|g\|_{\sL^2_{\alpha,\alpha}}=\Big(\int
(1+|\xi_1|^2)^{\alpha} (1+|\xi_2|^2)^{\alpha} |\widehat g(\xi)|^2 d\xi\Big)^{1/2} 
$$ is finite. Let $\alpha>1/2$ and $m$ be a bounded function such that
\Be\sup_{t_1>0, t_2>0} \|\eta_{\pr} \,m(t_1\cdot, t_2\cdot)\|_{\cL^2_{\alpha,\alpha}} \le B.
\label{Marccond} 
\Ee 
Then we have, for $1<p<\infty$, 
\Be\label{Marcthm} \|\cF^{-1}[m\widehat f]\|_p\le c_p B 
\|f\|_p.
\Ee
One can prove this using a straightforward product-type modification of Stein's proof of 
the Mikhlin-H\"ormander multiplier theorem in \cite[\S3]{Ste70}. One can also deduce it from R. Fefferman's theorem  \cite{F-R87}, cf. \cite{FLin}, \cite{CaSe92}.

\section{Some Maximal Function Estimates for Families 
	\\ of Mikhlin Type Multipliers on $\bbR^2$}\label{mikhlinmaxsect}

In this section we consider Mikhlin-H\"ormander multipliers with respect to the  dilation group 
$\delta^b_t$, $b>0$,
with $\delta_t^b(\xi)= (t\xi_1, t^b\xi_2)$.

\begin{thm}\label{SIO-thm}
	Suppose that
	\Be \label{hoermcond}
	\sup_{t>0} \sum_{|\alpha|\le 4} \big\|\partial^\alpha\big( \eta (\cdot) a(\delta^b_t\cdot)\big)\big\|_{L^1(\bbR^2)} \le 1
	\Ee
	Define, for $n\in \bbZ$ the operator $T_n$ by
	\Be
	\widehat {T_n f} (\xi)= a(\xi_1, 2^{bn}\xi_2)\widehat f(\xi) .
	\Ee
	Let $\cN$ be a subset of $\bbZ$ with $\# \cN=N$. Then for $1<p<\infty$,
	\Be \label {Lpconclusion}
	\Big\|\sup_{n\in \cN} |T_n f|\Big\|_p \le C_p  \sqrt{\log(1+N)} \|f\|_p.
	\Ee
\end{thm}

By the Marcinkiewicz interpolation theorem it suffices to show that there is $A=A(p)$ such that the inequality
\Be \label{wtineq}
\meas\big(\{x:  \sup_{n\in \cN} |T_n f|>4\lambda\}\big)
\le \big( A  \sqrt{\log(1+N)}  \la^{-1}  \|f\|_p\big)^p
\Ee holds for all Schwartz functions $f$ whose Fourier transform is compactly supported in $\bbR^2\setminus\{0\}$,  all $\la>0$ and all 
$\cN$ with $ \#\cN\le N$.

One can decompose
\Be a(\xi_1,\xi_2)=\sum_{j\in \bbZ} a_j (2^{-j}\xi_1,2^{-bj}\xi_2) \Ee
where 
each $a_j$ is supported in $\{(\xi_1,\xi_2): 1/2< |\xi_1|+ |\xi_2|^{1/b}< 2\}$ and 
$$ \sup_j\int\big|\partial_\xi^\alpha a_j(\xi) \big|\, d\xi \le C_\alpha, \quad|\alpha|\le 4.
$$ 
We shall repeatedly use that the operators $T_n$ are bounded on $L^p(\bbR^2)$ with norm 
independent of $n$. This follows by the Mikhlin-H\"ormander multiplier theorem and rescaling in the second variable.

Let  $\cT_\cN f:=\sup_{n\in \cN} |T_n f|$ and set 
\Be \label{defofepsN} \varepsilon_N :=  (\log (C_1N))^{-1/2}\Ee
where $C_1>c_1^{-1} $ with $c_1$ as in \eqref{eq:CWW85},
also $\eps_N<1/2$. 
 Since $f$ is a  Schwartz function, with  $\widehat f$ compactly supported in $\bbR^2\setminus \{0\}$ the function  $\cT_\cN f$ is in $L^\infty\cap L^2$ which allows us to apply the Chang-Wilson-Wolff inequality.

We have that
\beq\label{eq:H*set_est}
\begin{split}
	& \meas\big(\{x\in\R^2:\cT_\cN f(x)>4\lambda \} \big) \\
	&\quad \le \sum_{n\in\cN} \meas\big(\{x\in\R^2:|T_n f(x)|>4\lambda,\,\fS^{(2)} T_n f(x)\le \varepsilon_N \lambda \} \big)\\
	&\quad\quad + \meas\big( \big \{x\in\R^2 : \sup_{
		n\in \cN} |\fS^{(2)} [T_n f](x)| > \varepsilon_N \lambda \big \} \big).
\end{split}
\endeq

By the Chang-Wilson-Wolff inequality \eqref{eq:CWW85}, the first term on the right hand side of (\ref{eq:H*set_est}) is bounded by
\begin{align*}
& c_2 Ne^{-c_1\eps_N^{-2}}
\max_{n\in \cN} 
\meas\big(\big\{x\in \bbR^2: \mathcal{M}^{(2)}[T_nf]>\lambda\big\}\big) 
\\
&\le 
c_2 Ne^{-c_1\eps_N^{-2}}
\max_{n\in \cN}\lambda^{-p} 
 \|\cM^{(2)}[T_n f]\|_p^p 
\lc Ne^{-c_1\eps_N^{-2}}
\lambda^{-p} 
\|f\|_p^p
\lc \lambda^{-p} 
\|f\|_p^p
\end{align*}
where we used that $Ne^{-c_1\eps_N^{-2}}\le 1$ (by  \eqref{defofepsN}) and that 
the operators $T_n$ are uniformly bounded.

By Chebyshev's inequality  the  second term on the right hand side of (\ref{eq:H*set_est}) is bounded by
\begin{align*}
&\varepsilon_N^{-p} \lambda^{-p} \Big \|
\sup_{n\in \cN} \fS^{(2)}[T_n f]
\Big \|_{L^p}^p
\\
&\lc  \varepsilon_N^{-p} \lambda^{-p} \Big \|\sup_{n\in\cN} 
\Big(\sum_{k\in \bbZ} \big[
M^{(2)}( |T_n P^{(2)}_{k,b}  f|^q) ]^{2/q} \Big)^{1/2}
\Big\|_p^p\,.
\end{align*}

Here we have used Lemma \ref{martvsLP-lem}
with $g=T_n f$ 
and the fact that the operators  $T_n$  and $P^{(2)}_{k,b} $ commute; $q$ will be chosen so that $1 < q < p$.

We shall now use an idea in \cite{GHLR17} and approximate the operators $T_n $ by a convolution operator acting in the first variable. Define $T^{(1)} $ by
\begin{align*}
&\widehat {T^{(1)}  f} (\xi_1,\xi_2)= \sum_{j\in\bbZ} a_j(2^{-j}\xi_1,0) \widehat f(\xi_1,\xi_2).
\end{align*}
Recall the definition of $\chi_b$ in Lemma \ref{martvsLP-lem}. Notice also that \[a_j(2^{-j}\xi_1, 2^{(n-j)b}\xi_2)\chi_{b}(2^{-kb}\xi_2)\equiv 0\] if 
$j<n+k-1$ and therefore we have 
\begin{subequations}
	\begin{align}T_n P^{(2)}_{k,b}f
	&=  \sum_{j\ge n+k-1} \cF^{-1} [ a_j(2^{-j} \cdot, 2^{(n-j)b}\cdot)] * P^{(2)}_{k,b}f\notag
	\\
	\label{firstonevarterm} 
	&=  \sum_{j\ge n+k-1} \cF^{-1} [ a_j(2^{-j} \cdot, 0)] * P^{(2)}_{k,b}f 
	\\  \label{seconddiffterm}&+   \sum_{j\ge n+k-1} \cF^{-1} [ 
	a_j(2^{-j} \cdot, 2^{(n-j)b}\cdot)-
	a_j(2^{-j} \cdot, 0)]
	* P^{(2)}_{k,b}f.
	\end{align}
\end{subequations}
For the first term \eqref{firstonevarterm}  we use the 
one-dimensional version of Proposition  \ref{Cotlarprop}  to  get
\begin{align} \label{firstonevartermbd} 
\Big|\sum_{j\ge n+k-1} \cF^{-1} [ a_j(2^{-j} \cdot, 0)] * P^{(2)}_{k,b}f \Big| 
\lc M ^{(1)} (P^{(2)}_{k,b}f) +  M ^{(1)}( T^{(1)} P^{(2)}_{k,b}  f) .
\end{align} 
Here $M^{(1)}$ denotes the Hardy-Littlewood maximal operator acting on the first variable. 

Now consider  the second term \eqref{seconddiffterm}. Let $\tilde\phi$ be an appropriately chosen non-negative bump function supported in $(1/4,3)\cup (-3, -1/4)$ and let $K_{j,k,n}$ be the convolution kernel with multiplier
$$\widehat {K_{j,k,n}} (\xi) =\tilde\phi(2^{-kb}\xi_2)\big( 
a_j(2^{-j} \xi_1, 2^{(n-j)b}\xi_2) -
a_j(2^{-j} \xi_1, 0)\big).
$$
Then 
$$
\widehat {K_{j,k,n}} (2^j\xi_1, 2^{kb}\xi_2) =
2^{(k+n-j)b} 
\tilde\phi(\xi_2) \xi_2 \int_0^1 \partial_{2} a_j  (\xi_1, 2^{(k+n-j)b}s\xi_2) \, ds
$$
and we have 
$\big\|\partial^\alpha\big( \widehat {K_{j,k,n}} (2^j\cdot, 2^{kb}\cdot) \big)\big\|_1
\lc 2^{(k+n-j)b}$ for multiindices  $|\alpha|\le 3$. This implies 
\begin{align*} 
|K_{j,k,n}(x)|\lc 2^{(k+n-j)b} 
\frac{2^{j+kb}}{ (1+2^j|x_1|+ 2^{kb}|x_2|)^{3} }
\end{align*} and hence
$$\sum_{j\ge n+k-1} |K_{j,k,n} * P^{(2)}_{k,b} f(x)
|\lc M_{\text{str}}(P^{(2)}_{k,b} f) (x) $$
where 
$M_{\text{str}}$ is the strong maximal operator which is controlled by 
$M^{(2)}\circ M^{(1)}$.

Combining the estimates 
we thus see that
the 
second term on the right hand side of (\ref{eq:H*set_est}) is bounded by
\begin{align*}
\eps_N^{-p}\la^{-p} \biggl( 
&\Big \| \Big(\sum_{k\in \bbZ} \big[
M^{(2)}( |M^{(2)}M^{(1)} 
P^{(2)}_{k,b}  f|^q) ]^{2/q} \Big)^{1/2}
\Big\|_p
\\ 
&+
\Big \| \Big(\sum_{k\in \bbZ} \big[
M^{(2)}( |M^{(2)}M^{(1)} T^{(1)}  
P^{(2)}_{k,b}  f|^q) ]^{2/q} \Big)^{1/2}
\Big\|_p\biggr)^p .
\end{align*}
We use this with $1<q<p$ and apply  Fefferman-Stein estimates  for the vector-valued versions of 
$M^{(1)}$ and $M^{(2)}$ and the Marcinkiewicz-Zygmund theorem on $L^p(\ell^2) $ boundedness applied to the operator $T^{(1)}$.
Consequently  the last expression can be  bounded by 
$$C_p^p \eps_N^{-p}\la^{-p}\|f\|_p^p
\lc C_p^p (\log (1+N))^{p/2}\la^{-p}\|f\|_p^p\, ,
$$
by the definition of $\eps_N$. This finishes the proof of \eqref{wtineq} and thus the proof of Theorem \ref{SIO-thm}.
\qed

\section{Proof of Theorem \ref{Suthm}}\label{Suproofsect}
We decompose 
$  \Phi_0=\sum_{l\in \bbZ} \Phi_{0,l}$
where $\widehat \Phi_{0,l}(\xi)= \chi_+(2^{-l}|\xi| )\widehat \Phi_{0}(\xi)$ .
Define
\begin{align*}
a_{0,l} (\xi)&= \widehat \Phi_{0,l} (2^l \xi),
\\
\widetilde a_{0,l,s } (\xi)&= s^{b}\xi_2 \frac{\partial \widehat \Phi_{0,l}}{\partial\xi_2} (2^l \xi_1, 2^l s^b\xi_2 ).
\end{align*}
Then the functions $a_{0,l}$ and $\widetilde a _{0,l,s} $, for every $s\in (1/2,2)$, are supported in $\{\xi: 10^{-b}<|\xi|<10^b\}$  and satisfy the estimates
\[ \int\big|\partial_\xi^\alpha a_{0,l}(\xi)\big |d\xi +
\int\big|\partial_\xi^\alpha \widetilde a_{0,l,s}(\xi)\big|d\xi 
\le C 2^{-|l|}
\]
for all multiindices $\alpha$ with $|\alpha_1|+|\alpha_2| \le 10$.
This means that there is a $c>0$ such that  the multipliers
\Be \label{al-als-multipliers}
\begin{aligned} a_l(\xi) &= c2^{|l|} \sum_{j\in \bbZ}  a_{0,l} (2^{-j}\xi_1,2^{-jb}  \xi_2),
	\\
	\widetilde a_{l,s} (\xi) &= c2^{|l|} \sum_{j\in \bbZ}  a_{0,l,s} (2^{-j}\xi_1,2^{-jb}  \xi_2)
\end{aligned}
\Ee
satisfy the conditions 
\eqref{hoermcond} in Theorem \ref{SIO-thm}. Now define operators $S^u_l$ and $R^u_l$ 
\begin{align*} 
\widehat {S^{u}_l f}(\xi) &= \sum_{j\in \bbZ} 
\widehat {\Phi_{0,l}}(2^{-j}\xi_1, 2^{-jb}u\xi_2)
\widehat  f(\xi),
\\
\widehat {R^{u}_l f}(\xi) &= \sum_{j\in \bbZ} 
\widehat {\Phi_{0,l}}(2^{l-j}\xi_1, 2^{l-jb}u\xi_2)
\widehat f(\xi).
\end{align*}
The assertion of the theorem follows if we can prove
\Be\label{Sul-claim}
\big\| \sup_{u\in U} |S^{u}_{l} f| \big\|_p
\lc 2^{-|l|} \sqrt{\log \fN(U)} \|f\|_p
\notag
\Ee
which follows by isotropic rescaling  from
\Be\label{Rul-claim}
\big\| \sup_{u\in U} |R^{u}_{l} f| \big\|_p
\lc 2^{-|l|} \sqrt{\log \fN(U)} \|f\|_p.
\Ee

Now let $$\cN=\{n\in \bbZ: \exists s\in (1/2,2) \text{ such that } (2^{n}s)^{b} \in U\}.$$
Observe that $\#\cN \le C(b) \fN(U)$. The inequality \eqref{Rul-claim} follows from

\Be\label{Rul-stronger-claim}
\big\| \sup_{n\in \cN} \sup_{1/2<s<2} 
|R^{(2^{n}\!s)^b}_{l} f| \big\|_p
\lc 2^{-|l|} \sqrt{\log (1+\#\cN)} \|f\|_p
\notag
\Ee
which is a consequence of 
\Be\label{Rul-stronger-claim-lac}
\Big\| \sup_{n\in \cN}   |R^{2^{nb}}_{l} \!f
| \Big\|_p
\lc 2^{-|l|} \sqrt{\log (1+\#\cN)} \|f\|_p 
\Ee
and 
\Be\label{Rul-stronger-claim-der}
\int_{1/2}^{2} 
\Big\| \sup_{n\in \cN}  \big |\frac{\partial}{\partial s} R^{(2^{n}\!s)^b}_{l} \!f
\big| \Big\|_p\,ds\,
\lc 2^{-|l|} \sqrt{\log (1+\#\cN)} \|f\|_p \,. 
\Ee
Since 
\begin{align*} \cF[R^{2^{nb}}_{l}\! f](\xi)&= \sum_j a_{0,l}(2^{-j}\xi_1, 2^{nb-jb}\xi_2) \widehat f(\xi),\\
\cF[\partial_s R^{(2^{n}\!s)^b}_{l} \!f
](\xi) &= \frac {b}{s} \sum_j a_{0,l,s} (2^{-j}\xi_1, 2^{nb-jb}\xi_2) \widehat f(\xi),
\end{align*} 
the  inequalities  \eqref{Rul-stronger-claim-lac} and 
\eqref{Rul-stronger-claim-der}
follow by applying Theorem \ref{SIO-thm} to the multipliers 
in \eqref{al-als-multipliers}. \qed

\section{Proof of Theorem \ref{Tuthm}}\label{Tuproofsect}
We only consider the maximal function for the operator $T^u_+$, since the analogous problem for $T^u_-$ can be reduced to the former one by a change of variable
(with a different curve). We omit the subscript and set $T^u=T^u_+$.

Decompose $\kappa_{0,+}=\sum_{\ell=0}^{\infty} \kappa_{0,\ell}$ where
$$
\widehat {\kappa_{0,\ell}}(\xi) = \chi_+(2^{-\ell }|\xi|) \omega_+(\xi) e^{i\Psi_+(\xi)}.
$$
Notice that, by Lemma \ref{symbollemma},  $|\xi_1|\approx|\xi_2|\approx 2^{\ell}$ for $ \xi\in \supp(\widehat {\ka_{0,\ell}})$.
Define $\kappa_{j,\ell}$ by $\widehat {\ka_{j,\ell}} (\xi)= \widehat{\kappa_{0,\ell}} (2^{-j}\xi_1,2^{-jb}\xi_2) $ and define $T^u_{j,\ell}$ by
\Be \label{Tujelldef}
\widehat {T^u_{j,\ell} f}(\xi) = \widehat{\kappa_{j,\ell} }(\xi_1,u \xi_2) \widehat{f}(\xi).
\Ee
Then we have $T^u=\sum_{\ell \ge 0} \sum_{j\in \bbZ} T^u_{j,\ell}$.

The assertion of the theorem follows if we can show, for $2 < p < \infty$, that there exists some $\eps=\eps(p)>0$ with
\Be \label{T-ell-assertion} 
\Big \|  \sup_{n\in \bbZ} \sup_{1/2<s<2} \Big|\sum_{j\in \bbZ} 
T_{j,\ell}^{(2^{n}s)^{b}} \!f\Big|\,\Big\|_p\lc 2^{-\ell \eps} \|f\|_p.
\Ee

Define  $\cR^u_{j,\ell} $ by 
$$
\widehat {\cR^u_{j,\ell} f}(\xi) = \widehat{\kappa_{0,\ell} }(2^{\ell -j}\xi_1,2^{\ell -jb}u \xi_2) \widehat{f}(\xi).
$$
By isotropic rescaling
inequality \eqref{T-ell-assertion}  is equivalent with 
\Be \label{R-ell-assertion} 
\Big \|  \sup_{n\in \bbZ} \sup_{1/2<s<2} \big| \sum_{j\in \bbZ} 
\cR_{j,\ell} ^{(2^{n}s)^{b}} \!f\big|\Big\|_p\lc 2^{-\ell \eps} \|f\|_p.
\Ee
This inequality follows, by the embedding $\ell^p\subset \ell^\infty$ and Fubini's theorem  from
\Be\label{Rellmax}
\Big(\sum_{n\in \bbZ} 
\Big \| \sup_{ 1/2< s<2} \big| \sum_{j\in \bbZ}
\cR_{j,\ell}^{(2^n s)^{b}}\!f
\big|\Big\|_p^p\Big)^{1/p} \lc 2^{-\ell \eps} \|f\|_p
\Ee
Fix $n,x$ and set $G(s)=  \sum_j \cR_{j,\ell} ^{(2^n s)^{b}} \!f(x)$. We use the standard argument of applying the fundamental theorem of calculus to $|G(s)|^p$ and then H\"older's inequality  which gives
\[|G(s)|^p \le  |G(1)|^p + p \Big(\int_{1/2}^2
|G(s) |^pds\Big)^{1/p'} \Big(\int_{1/2}^2 |G'(s) |^pds\Big)^{1/p} .
\]
This inequality and another application of H\"older's inequality in $\bbR^2$ shows that \eqref{Rellmax} follows from
\begin{subequations} 
	\Be
	\label{Relllsm}
	\Big(\sum_{n\in \bbZ} \int_{1/2}^2
	\Big \|  \sum_j
	\cR_{j,\ell}^{(2^n s)^{b}}\!f
	\Big\|_p^pds\Big)^{1/p} \lc 2^{-\ell (\eps+1/p)} \|f\|_p,
	\Ee
	\Be 
	\label{Relllsmder}
	\Big(\sum_{n\in \bbZ} \int_{1/2}^2
	\Big \| \frac{\partial}{\partial s} \Big(\sum_j
	\cR_{j,\ell} ^{(2^n s)^{b}}\!f\Big)
	\Big\|_p^pds\Big)^{1/p} \lc 2^{\ell -\ell (\eps+1/p)} \|f\|_p
	\Ee
	and
	\Be 
	\label{Rellfixed}
	\Big(\sum_{n\in \bbZ} \Big \|  \sum_j \cR_{j,\ell}^{2^{nb}}\!f
	\Big\|_p^p \Big)^{1/p} \lc 2^{-\ell/p} \|f\|_p
	\Ee
\end{subequations}
for $2 < p < \infty$.

We focus on the derivation of the inequality \eqref{Relllsm}. Note that for $s\in [1/2,2]$ 
\begin{align*} \widehat {\kappa_{0,\ell}}(\xi_1, s^b \xi_2)
&= \omega_+(\xi_1, s^b\xi_2) \chi_+(2^{-\ell} |(\xi_1, s^b \xi_2)| )
e^{i \Psi_+(\xi_1, s^b\xi_2)} 
\\&= 
2^{-\ell/2} \eta_{\ell,s}(2^{-\ell}\xi)
e^{-is^{\frac{b}{b-1}}\Psi_+(\xi_1, \xi_2)}
\end{align*}
where 
$$\eta_{\ell,s}(\xi_1,\xi_2)
= 2^{\ell/2} 
\omega_+(2^{\ell} \xi_1, 2^{\ell}  s^b\xi_2)  \chi_+(|(\xi_1, s^b \xi_2)| )
$$
and taking into account that $\om_+$ is a symbol of order $-1/2$ we see that the 
 $\eta_{\ell,s}$ belong to a bounded set of $C^\infty$ functions supported in an annulus $\{\xi: a_0\le |\xi|\le a_0^{-1}\}$,  for fixed $a_0=a_0(b)<1$.

After changing variables $t=s^{-\frac b{b-1}}$,  with $t\in (2^{-\frac{b}{b-1}} ,2^{\frac{b}{b-1}}) $
this puts us in the position to apply \eqref{lsineq}
 with $R=2^{\ell }$ 
and we obtain, with suitable 
$\eps'=\eps'(p)>0$ 
\[
\Big( \int_{1/2}^2
\Big \| \F^{-1}[ \widehat {\kappa_{0,\ell}} (\xi_1, s^b \xi_2) 
\widehat f]\Big\|_p ^p
ds\Big)^{1/p} \lc 2^{-\ell (\eps'+1/p)} \|f\|_p.
\] 
By isotropic scaling, replacing 
$ \widehat {\kappa_{0,\ell}} (\xi_1, s^b \xi_2) $ with 
$ \widehat {\kappa_{0,\ell}} (2^\ell \xi_1, s^b 2^\ell \xi_2) $,   we also have 
\Be \label{assa-of-vectthm}
\Big(\int_{1/2}^2
\big \| 
\cR_{0,\ell} ^{s^b}f
\big\|_p ^p
ds\Big)^{1/p} \le C_\eps 2^{-\ell (\eps'+1/p)} \|f\|_p.
\Ee

Let $$m_{j, \ell} (\xi, s) = 
\widehat {\kappa_{0,\ell} }(2^{\ell-j}  \xi_1, s^b 2^{\ell -jb}  \xi_2) $$  
and observe  $\widehat{\cR_{j,\ell} ^{s^{b}}f} (\xi) = m_{j, \ell} (\xi, s) \widehat f(\xi)$.
The  functions  $\xi \mapsto m_{0,\ell}(\xi,s) $ 
are supported in a fixed annulus  and satisfy 
\Be \label{assb-of-vectthm}
|\partial_{\xi_1}^{\alpha_1 }
\partial_{\xi_2}^{\alpha_2}
m_{0,\ell}(\xi,s) | \lc 2^{\ell(\alpha_1+\alpha_2)} .
\Ee By  Corollary 
\ref{vectmultthm} we get the inequality
\begin{multline} \label{eq:vvls}
\Big(\int_{1/2}^2 \Big\| \Big(\sum_{j\in \bbZ} | \cR^{s^{b}}_{j,\ell} f_j  |^2\Big)^{1/2} \Big\|_p^p ds\Big)^{1/p} \\
\lc  2^{-\ell(\eps'+1/p)} (1+\ell)^{1/2-1/p} 
\Big\|\Big(\sum_j|f_j|^2\Big)^{1/2}\Big\|_p.
\end{multline}

We can replace the multipliers 
$m_{j, \ell }(\xi_1,\xi_2,s) $ by 
$m_{j, \ell }( \xi_1,2^{nb} \xi_2, s) $, after scaling  in the second variable. 
This means that for every fixed $n$ we have proved, for $\eps<\eps'$,
\Be \label{vectvallsmprel}
\Big( \int_{1/2}^2
\Big \| \Big(\sum_j|
\cR_{j,\ell}^{(2^n s)^{b}}
\!f_j|^2 \Big)^{1/2} 
\Big\|_p^pds\Big)^{1/p} \lc 2^{-\ell (\eps+1/p)}  \Big\|\Big(\sum_j |f_j|^2\Big)^{1/2}\Big\|_p,
\Ee
with the implicit constant independent of $n$.

We now combine this with Littlewood-Paley inequalities to prove  \eqref{Relllsm}.
Let $\widetilde\chi^{(1)} $ be an even $C^\infty$ function supported on 
$\{\xi_1: |c_+|b 2^{-3b-1} \le |\xi_1|\le |c_+|b 2^{3b+1}\}$
and equal to $1$ for 
$ |c_+|b 2^{-3b} \le |\xi_1|\le |c_+|b 2^{3b}$.
Let $\widetilde\chi_b^{(2)} $ be an even $C^\infty$ function supported on 
$\{\xi_2: 2^{-2b-1} \le |\xi_2|\le2^{2b+1}\}$
and equal to $1$ for 
$ 2^{-2b} \le |\xi_2|\le2^{2b}$.
Define $\tilde P^{(1)}_{j} $, $\tilde P^{(2)}_{j,b} $ by
\begin{align*}
\widehat{\tilde P^{(1)}_{j} f}(\xi)&= \widetilde \chi^{(1)} (2^{-j}\xi_1) \widehat f(\xi)
\\
\widehat{\tilde P^{(2)}_{j,b} f}(\xi)&= \widetilde \chi^{(2)} (2^{-jb}\xi_2) \widehat f(\xi)
\end{align*} 
Then  by the support properties of $\widehat {\ka_{0,\ell}}(2^\ell\cdot)$  we get for  $1/2\le s\le 2$ 
\Be \label{Rjln-Fourierloc}
\cR_{j,\ell}^{(2^n s)^{b}}=
\tilde P^{(1)}_{j} 
\tilde P^{(2)}_{j-n,b} 
\cR_{j,\ell}^{(2^n s)^{b}}
\tilde P^{(2)}_{j-n,b} \tilde P^{(1)}_{j} .
\Ee
Hence, by Littlewood-Paley theory 
\begin{align*}
&\Big(\sum_{n\in \bbZ} \int_{1/2}^2
\Big \|  \sum_j
\cR_{j,\ell}^{(2^n s)^{b}}\!f
\Big\|_p^pds\Big)^{1/p} 
\\
&\lc
\Big(\sum_{n\in \bbZ} \int_{1/2}^2
\Big \|  \Big( \sum_j|
\cR_{j,\ell}^{(2^n s)^{b}}  \tilde P^{(2)}_{j-n,b} \tilde P^{(1)}_{j} f|^2\Big)^{1/2}
\Big\|_p^pds\Big)^{1/p} 
\end{align*} and by \eqref{vectvallsmprel} this is controlled by
\[ 
2^{-\ell(\eps(p) +1/p)} \Big(\sum_{n\in \bbZ} 
\Big \|  \Big( \sum_{j\in \bbZ}
|
\tilde P^{(2)}_{j-n,b} \tilde P^{(1)}_{j} f|^2\Big)^{1/2}
\Big\|_p^p\Big)^{1/p} 
\]
for some $\eps(p)>0$ when $2<p<\infty.$
We finish the proof of \eqref{Relllsm} 
by observing that 
\begin{align*}
\Big(\sum_{n\in \bbZ} 
\Big \|  \Big( \sum_{j\in \bbZ}|
\tilde P^{(2)}_{j-n,b} \tilde P^{(1)}_{j} f|^2\Big)^{1/2}
\Big\|_p^p\Big)^{1/p} \le 
\Big\|
\Big( \sum_{j\in \bbZ}\sum_{n\in \bbZ} |
\tilde P^{(2)}_{j-n,b} \tilde P^{(1)}_{j} f|^2\Big)^{1/2}
\Big\|_p &
\\=\Big \|  \Big( \sum_{k_1\in \bbZ} \sum_{k_2\in \bbZ} |
\tilde P^{(2)}_{k_2,b} \tilde P^{(1)}_{k_1} f|^2\Big)^{1/2}
\Big\|_p \lc \|f\|_p&
\end{align*} where we have used the embedding $\ell^2\hookrightarrow \ell^p$ for $p>2$, and applied a  two-parameter
Littlewood-Paley inequality.

We now turn to the estimate  \eqref{Relllsmder}.
A computation shows 
\begin{subequations} \begin{multline}
	2^{-\ell} \frac{\partial}{\partial s} \Big( \sum_j \cF[\cR_{j,\ell}^{(2^n s)^{b} }f](\xi)\Big)
	\\= \widehat f(\xi)  \frac{b}{s}  \sum_j \upsilon_\ell (2^{-j}\xi_1 , s^b 2^{(n-j)b} \xi_2) 
	e^{i2^\ell \Psi_+(2^{-j}\xi_1, s^b 2^{(n-j) b} \xi_2)}
	\end{multline}
	where 
	\begin{multline}\label{upsilonell} \upsilon_\ell(\xi)= 2^{-\ell}\chi_+'(|\xi|) \frac{\xi_2^2} {|\xi|} 
	\om_+(2^\ell\xi_1, 2^{\ell}\xi_2) 
	+\chi_+(|\xi|) \xi_2\frac{\partial \om_+}{\partial{\xi_2}}(2^\ell\xi_1, 2^{\ell}\xi_2) 
	\\+ \chi_+(\xi)\om_+(2^{\ell}\xi_1, 2^{\ell}\xi_2) i\xi_2\frac{\partial \Psi_+}{\partial{\xi_2}}(\xi_1, \xi_2).
	\end{multline}  
\end{subequations}
Here the main contribution in \eqref{upsilonell} comes from the third term 
(the others are similar but better by a factor of about $2^{-\ell}$).

It is now straightforward to check that in the proof of  \eqref{Relllsm} 
the term
$\cR_{j,\ell}^{(2^n s)^{b} }\!f$ can be replaced with
$2^{-\ell} \partial_s(\cR_{j,\ell}^{(2^n s)^{b} }\!f)$ and  one obtains
\eqref{Relllsmder}. 

Finally, a simple modification of the proof of (\ref{Relllsm}) would also prove (\ref{Rellfixed}): in place of \eqref{lsineq}, one would use a fixed time estimate, as stated immediately before \eqref{lsineq}.
This finishes the proof of Theorem \ref{Tuthm}.

\section{Maximal functions for lacunary sets} \label{lacunary-section}

We shall prove some upper bounds for the operator norm 
of $\cH^U$ for {\it lacunary} sets.

\noi{\it Definition.} Let $\ka>1$. 
A finite set $U$ is called $\ka$-lacunary if it can be arranged in a sequence 
$U=\{u_1<u_2<\dots<u_M\}$ where  $u_{j+1}\le  u_j/\ka$ for $j=1,\dots, M-1$.
$U$ is lacunary if $U$ is $\ka$-lacunary for some $\ka>1$.

\smallskip

Note that for lacunary sets we have $\# U \approx \fN(U)$ (with the implicit constant depending on $\ka$).

\begin{prop}\label{lacthm} Let 
	$U$ be a lacunary set.
	Then, for $4/3<p<\infty$
	\Be \label{lac43est}
	\|\cH^U\|_{L^p\to L^p}\lc  
	\sqrt{\log (1+(\#U))}\,.
	\Ee
\end{prop}

Proposition \ref{lacthm} will be used in the proof of lower bounds in \S\ref{karagulyan-sect}. For  this application it is important that \eqref{lac43est} just holds for some $p<2$. 
We do not know at this time whether the result extends to all $p>1$. \footnote{Added in September 2019: After the submission of this paper the  authors showed the bound of Proposition
\ref{lacthm}  for general lacunary sets $U$, in the  full range $1<p<\infty$.  This result can be found in the paper \cite{grsy2} which also  contains $L^p$ results,  $p<2$,  for more general sets $U$, under  suitable dimension assumptions.} 
For special lacunary sequences it does:
\begin{prop}\label{lacthmspecial} Let $U$ be a  subset 
	of $\{2^{nb}: n\in \bbZ\}$.
	Then, for $1<p<\infty$
	$$
	\|\cH^U\|_{L^p\to L^p}\lc  
	\sqrt{\log (1+(\#U))}\,.
	$$
\end{prop}
Here $b$ is as in the definition of the curve $\gamma_b$ in \eqref{gammabdef}.

\subsection{ \it Proof of Proposition \ref{lacthm}}
We may assume that for every interval $I_n:=[2^{nb}, 2^{(n+1)b})$, $n\in \bbZ$,  there is at most one $u\in U\cap I_n$. This is because of the lacunarity assumption we can split $U$ in $O(1)$ many sets with this assumption.

We order $U=\{u_\nu\}$ such that $u_{\nu}<u_{\nu+1}$ and  let $n(\nu)$ be the unique integer for which $u_\nu\in I_n$.

We split $H^{(u)}= S^{u}+T^u$ as in \eqref{hilbtransformdec}. In view of  Theorems \ref{Suthm}, \ref{Tuthm}  it suffices to prove the inequality
\Be\label{Tulacest} \big\|\sup_{u\in U} |T^u_\pm f| \big\|_p\lc \|f\|_p
\Ee for $4/3<p\le 2$. 
By the reduction in  \S\ref{Tuproofsect} this can be accomplished if 
\Be
\label{Rell-lacLp}
\Big \| \sup_{\nu} \big|\sum_j
\cR_{j,\ell}^{u_\nu}f\big|
\Big\|_p \lc 2^{-\ell \ep(p)} \|f\|_p
\Ee can be proved 
for $\ep(p)>0$, in our case in the range $4/3<p\le 2$.

Replacing the $\sup$ by an $\ell^2$ norm we see that 
\eqref{Rell-lacLp} follows from
\Be
\label{Rell-lacLpl2}
\Big \| \Big(\sum_{\nu} \big|\sum_j
\cR_{j,\ell}^{u_\nu}f\big|^2\Big)^{1/2}
\Big\|_p \lc 2^{-\ell \ep(p)} \|f\|_p
\Ee
Analogously  to 
\eqref{Rjln-Fourierloc} 
we have
$$\cR_{j,\ell}^{u_\nu}=
\tilde P^{(1)}_{j} 
\tilde P^{(2)}_{j-n(\nu),b} 
\cR_{j,\ell}^{u_\nu}
\tilde P^{(2)}_{j-n(\nu),b} \tilde P^{(1)}_{j} $$
and thus, by Littlewood-Paley theory, \eqref{Rell-lacLpl2} is a consequence of 
\Be
\label{Rell-lacLpl2l2}
\Big \| \Big(\sum_\nu\sum_{j\in \bbZ} \big|
\cR_{j,\ell}^{u_\nu} 
\tilde P^{(2)}_{j-n(\nu),b} \tilde P^{(1)}_{j} \!
f\big|^2\Big)^{1/2}
\Big\|_p \lc 2^{-\ell \ep(p)} \|f\|_p.
\Ee
By a standard application of Khintchine's inequality 
this estimate follows if we can prove 
\Be
\label{Rell-lacLpl2l2Kh}
\Big \| \sum_\nu\sum_{j\in \bbZ} c(\nu,j) \cR_{j,\ell}^{u_{\nu} }
\tilde P^{(2)}_{j-n(\nu),b} \tilde P^{(1)}_{j} \!
f\Big\|_p \lc 2^{-\ell \ep(p)} \|f\|_p.
\Ee
for an arbitrary choice of $\{c(\nu, j)\}$ with 
$\sup_{j,\nu} |c(\nu, j)|\le 1$.
Let $$\omega_\ell(\xi)= \om_+(2^\ell \xi ) \chi_+(|\xi|) $$ then $\omega_\ell$ and its derivatives are $O(2^{-\ell/2})$, by the symbol property of $\om_+$, and are supported on a common annulus.
We see that the $L^2$ operator norms of the individual operators 
$\cR_{j,\ell}^{u_\nu}$ are $O(2^{-\ell/2})$, and 
that 
the function 
\begin{multline*}m_\ell(\xi)= \sum_{\nu}\sum_j \tilde\chi^{(1)}(2^{-j}\xi_1)
\tilde \chi^{(2)}(2^{-jb+n(\nu)b}\xi_2) \times \\\om_\ell(2^{-j}\xi_1, 2^{(n(\nu)-j) b}\xi_2) e^{i2^\ell \Psi_+(2^{-j}\xi_1,2^{(n(\nu)-j)b}\xi_2)} 
\end{multline*}
has $L^\infty$ norm $\lc 2^{-\ell/2}$. This implies 
\Be
\label{Rell-lacLpl2l2Khp=2}
\Big \| \sum_\nu\sum_{j\in \bbZ} c(\nu,j) \cR_{j,\ell}^{u_{\nu} }
\tilde P^{(2)}_{j-n(\nu),b} \tilde P^{(1)}_{j} 
f\Big\|_2 \lc 2^{-\ell/2} \|f\|_2.
\Ee
For $p$ near $1$ we apply the Marcinkiewicz multiplier theorem in the form described in \S\ref{marcsect}. 
It is not hard to check that the multiplier $m_\ell$ satisfies the  condition  \eqref{Marccond} with constant $B\le C_\alpha 2^{\ell (2\alpha-1/2)}$. Hence we get 
\Be
\label{Rell-lacLpl2l2Khp>1}
\Big \| \sum_\nu\sum_{j\in \bbZ} c(\nu,j) \cR_{j,\ell}^{u_{\nu} }
\tilde P^{(2)}_{j-n(\nu),b} \tilde P^{(1)}_{j} \!
f\Big\|_p \lc 2^{\ell(2\alpha-\frac 12)} \|f\|_p, \quad \alpha>1/2.
\Ee
We interpolate between \eqref{Rell-lacLpl2l2Khp=2} and \eqref
{Rell-lacLpl2l2Khp>1}. By choosing $\alpha$ very close to $1/2$, we obtain \eqref{Rell-lacLpl2l2Kh} for any $p\in (4/3,2]$. \qed

\subsection{ \it Proof of Proposition \ref{lacthmspecial}} 
We argue as in the proof of Proposition \ref{lacthm}. The desired conclusion follows if under our present conditions \eqref{Rell-lacLpl2l2Khp>1} can be upgraded to
\Be
\label{Rell-lacLpl2l2Khp>1upgrade}
\Big \| \sum_\nu\sum_{j\in \bbZ} c(\nu,j) \cR_{j,\ell}^{u_{\nu} }
\tilde P^{(2)}_{j-n(\nu),b} \tilde P^{(1)}_{j} 
f\Big\|_p \le c_p (1+\ell^4) \|f\|_p, \quad 1<p\le 2.
\Ee
As now $u_\nu= 2^{n(\nu)b} $ for a strictly  increasing sequence $\{n(\nu)\}$ we see by  another application of Littlewood-Paley theory  that \eqref{Rell-lacLpl2l2Khp>1upgrade} is a consequence of the inequality
\Be
\label{Rell-lacLpl2vec}
\Big \| \Big(\sum_{n\in \bbZ}\sum_{j\in \bbZ} \big|
\cR_{j,\ell}^{2^{nb}}
f_{j,n}\big|^2\Big)^{1/2}
\Big\|_p \lc (1+\ell^4) \Big\|\Big(\sum_{j,n} |f_{j,n} |^2 \Big)^{1/2}\Big\|_p.
\Ee

This is proved  as in \cite{GHLR17} by using a superposition of shifted maximal operators, in a vector-valued setting. To analyze the situation 
we recall how $\cR_{j,\ell}^{u}$ was formed
(namely by rescaling $T_{j,\ell}^{u}$, then see \S\ref{deccomp-sect}).

Let $\sigma_+$ be as in \eqref{sigmapmdef}. Then there is a Schwartz function $\varsigma$ 
such that
\begin{align*}
\widehat {\cR^{2^{nb}}_{j,\ell} f}(\xi) 
= \chi_+(|(2^{-j}\xi_1, 2^{nb-jb}\xi_2)|)
\widehat{\sigma_+} (2^{\ell-j}\xi_1, 2^{\ell+nb-jb} \xi_2) \widehat f(\xi)&\\+
\chi_+(|(2^{-j}\xi_1, 2^{nb-jb}\xi_2)|) \widehat{\varsigma} (2^{\ell-j}\xi_1, 2^{\ell+nb-jb} \xi_2) \widehat f(\xi)&.
\end{align*}
Consider the second (error) term. It is easy to see
that
\[
\big|\cF[ \chi_+(|(2^{-j}\xi_1, 2^{nb-jb}\xi_2)|) \widehat{\varsigma} (2^{\ell-j}\xi_1, 2^{\ell+nb-jb} \xi_2) \widehat f(\xi)](x)\big| \lc 
2^{-\ell} 
M_{\text{str} } f(x)
\]
so that these terms are taken care of by an application of the Fefferman-Stein inequality for the vector-valued strong maximal function.

We concentrate on the main term.
We write $\sigma_+=\sum_{m=2^{\ell-1}}^{2^{\ell+1} } \mu_{m}$ where the measure $\mu_m$ is given by

$$\inn{\mu_m}{f}= \int_{m2^{-\ell}}^{(m+1)2^{-\ell}} 
f(t,\gamma_b(t))\chi_+(t) \frac{dt}{t}.
$$

Define $\cR^{u}_{j,\ell,m} f$ by 
\[\widehat {\cR^{u}_{j,\ell,m} f}(\xi) 
= \chi(|(2^{-j}\xi_1, 2^{-jb}\xi_2)|)\widehat{\mu_m} (2^{\ell-j}\xi_1, 2^{\ell-jb}u \xi_2) \widehat f(\xi).\]
Then by the above discussion we have 
\[\Big | \cR^{2^{nb}}_{j, \ell} f(x)
- \sum_{m=2^{\ell-1}}^{2^{\ell+1} } \cR^{2^{nb}}_{j, \ell,m} f(x)\Big|\lc 2^{-\ell}
M_{\mathrm{str}} f(x) 
\] 
and hence, by Minkowski's inequality,  it suffices to show that 
\Be
\label{Rell-lacL1l2vecm}
\Big \| \Big(\sum_{n,j\in \bbZ} \big|
\cR_{j,\ell,m}^{2^{nb}}
f_{j,n}\big|^2\Big)^{1/2}
\Big\|_p \lc 2^{-\ell} (1+\ell)^4  \Big\|\Big(\sum_{j,n\in\Z} |f_{j,n} |^2 \Big)^{1/2}\Big\|_p
\Ee
for $2^{\ell-1}\le m\le 2^{\ell+1}$.
Notice that
\begin{multline*}|\mu_m* \cF^{-1} [\chi_+(|\cdot|2^{-\ell})](y)\big|\\
\lc 2^{-\ell}  \frac{2^\ell} {
	(1+2^\ell|y_1-m2^{-\ell}|)^{10}} 
\frac{ 2^{\ell} }
{(1+2^\ell|y_2-m^b2^{-\ell b}|)^{10}}
\end{multline*}

Now define
\begin{align*}
\rho_{m,k_1}^{(1)}(y_1)&= 2^{k_1}(1+ |2^{k_1}y_1-m|)^{-10}
\\
\rho_{m,k_2}^{(2)}(y_2)&= 2^{bk_2 }
(1+ 2^{bk_2} |y_2-m^b 2^{-\ell(b-1)} |)^{-10}
\end{align*} 
We then have the pointwise estimate
\Be \label{pointwiseshifted}
|\cR^{2^{nb}}_{j, \ell} f(x)| \lc 2^{-\ell} (\rho_{m,j}^{(1)}\otimes \rho_{m, j-n}^{(2)})*|f|.
\Ee
By an application of inequalities for the shifted maximal operators (see \cite[Theorem 3.1]{GHLR17}) we see that the expressions 
\begin{subequations} 
	\begin{align}
	&\Big(\int\Big| \Big( \sum_{k_1, k_2} \big[ \int \rho_{m,k_1}^{(1)} (x_1-y_1)  |g_{k_1,k_2} (y_1,x_2)|dy_2\Big]^2\Big)^{p/2}dx\Big)^{1/p} ,
	\notag
	\\
	&\Big(\int\Big| \Big( \sum_{k_1, k_2} \Big[ \int \rho_{m,k_2}^{(2)} (x_2-y_2)  |g_{k_1,k_2} (x_1,y_2)|dy_2\Big]^2\Big)^{p/2}dx\Big)^{1/p} 
	\notag
	\end{align}
\end{subequations}
are both bounded by a constant times
\[ (\log m)^2
\Big\|\Big(\sum_{k_1,k_2}|g_{k_1, k_2} |^2\Big)^{1/2} \Big\|_p.
\notag
\]
Applying both estimates iteratively we get 
\[
\Big\| \Big(\sum_{k_1,k_2} [ 
(\rho_{m,k_1}^{(1)}\otimes \rho_{m, k_2}^{(2)})*|g_{k_1,k_2}|]^2\Big)^{1/2} \Big\|_p \lc 
(\log m)^4
\Big\|\Big(\sum_{k_1,k_2}|g_{k_1, k_2} |^2\Big)^{1/2} \Big\|_p.
\]
We apply this with  $g_{k_1,k_2}= f_{k_1, k_1-k_2}$
and use \eqref{pointwiseshifted} to obtain 
\eqref{Rell-lacL1l2vecm}. \qed
\section{Lower bounds}\label{karagulyan-sect}

\subsection{\it The main lower bound and some consequences} The purpose of this section is  to prove 
the lower bound
\begin{thm} \label{lowerboundtheorem}
	Let $U\subset (0,\infty)$ and $1<p<\infty$.
	Then there is a constant $c_p$ such that 
	$$ \|\cH^U\|_{L^p\to L^p} \ge c_p \sqrt{\log(\fN(U))}.$$
\end{thm}

\subsubsection{Some consequences}

(i) First, Theorem \ref{lowerboundtheorem} in combination with the  already proven upper bounds in Theorems \ref{Suthm} and \ref{Tuthm} yields the equivalence 
(with constants depending on $p$)
\Be\label{HUequiv} \|\cH^U\|_{L^p\to L^p} \approx \sqrt{\log (\fN(U))}
\Ee for $2 < p < \infty$, stated as Theorem \ref{hil-max}.

(ii) We also immediately get an equivalence in  Propositions \ref{lacthm} and \ref{lacthmspecial} which we formulate as
\begin{cor} Let $U$ be a lacunary set. 
	Then  \eqref{HUequiv} holds for $4/3<p<\infty$. 
	If $U$ is contained in $\{2^{nb}: n\in \bbZ\} $ then 
	\eqref{HUequiv}  holds for $1<p<\infty$.
\end{cor}

\subsubsection{\it Reduction to the case $p=2$}
\label{p=2redux}
Let $U_*$ be a maximal subset of $U$ with the property that each interval $[2^n, 2^{n+1}]$ contains 
at most one point in $U$. Then  $\#(U_*)\approx \fN(U)$.  Let $\widetilde U$ be any finite subset of $U_*$ with the understanding that $\widetilde U=U_*$ if $U_*$ is already finite.
Clearly 
$$\|\cH^U\|_{L^p\to L^p} \ge  \|\cH^{U_*} \|_{L^p\to L^p} \ge 
\|\cH^{\tilde U}\|_{L^p\to L^p}
$$
and thus it suffices to prove the inequality 
\Be \label{Utildeineq}
\|\cH^{\tilde U}\|_{L^p\to L^p} 
\gc A_p \sqrt{\log (\#\widetilde U) }.
\Ee

We show that it suffices to prove 
\eqref{Utildeineq} for $p=2$: 
Since $\widetilde U$ is a  disjoint union of two lacunary sets we have the inequality
\Be \notag
\|\cH^{\tilde U}\|_{L^q\to L^q} \le C_q \sqrt{\log (\#\widetilde U) }, \quad \text{ for } 4/3<q<\infty,
\Ee by Proposition \ref{lacthm}.

If $1<p<2$ we pick $q$ such that $2<q<\infty$, and if $2<p<\infty$ we pick $q$ such that $4/3<q<2$.
Let $\theta\in (0,1)$ such that $(1-\theta)/p+\theta/q=1/2$.
We have
\begin{align*} A_2 \big(\log (\#\widetilde U) \big)^{1/2}
&\le \|\cH^{\tilde U}\|_{L^2\to L^2}
\le \|\cH^{\tilde U}\|_{L^p\to L^p}^{1-\theta}
\|\cH^{\tilde U}\|_{L^q\to L^q}^\theta
\\&\le \big(c_q (\log (\#\widetilde U))^{1/2} )^\theta
\|\cH^{\tilde U}\|_{L^p\to L^p}^{1-\theta}
\end{align*} which implies 
$$
\|\cH^{\tilde U}\|_{L^p\to L^p} \ge A_2^{\frac{1}{1-\theta}} c_q^{-\frac{\theta}{1-\theta}} \sqrt{\log (\#\widetilde U) }.
$$

For the remainder of this section we shall verify 
the lower bound in \eqref{Utildeineq} for $p=2$. 
We shall need to skim the set $\widetilde U$ a bit more.
To prepare for this we first  study in more detail the multipliers of the Hilbert transforms.

\subsection{\it Observations on the multipliers for the Hilbert transforms}
We may assume $c_+>0$. We write
$\widehat {H^{(u)} f} (\xi)= m(\xi_1,u\xi_2) \widehat f(\xi)$ where
\[ m(\xi_1,\xi_2) = \lim_{\substack {\eps\to 0+\\ R\to \infty} }\Big(
\int_{\eps<t\le R} 
e^{-i (t\xi_1+ c_+ t^b\xi_2)} \frac {dt}{t} +
\int_{-R<t<-\eps} 
e^{-i (t\xi_1+ c_- (-t)^b\xi_2)} \frac {dt}{t} \Big).
\]
By the homogeneity of the curve $\Gamma_b$ with respect to the dilations $(\xi_1,\xi_2) \mapsto (\lambda \xi_1, \lambda^b \xi_2)$, we see that
$m(\la\xi_1, \la^b\xi_2)=m(\xi_1,\xi_2) $ for $\la>0$.
Moreover one can check that
$m$ is continuous on $\R^2 \setminus \{0\}$, 
\begin{subequations}
	\Be \label{xi2=0}
	m(\xi_1,0) = -\pi i \,\sgn \xi_1, \quad \xi_1 \ne 0,
	\Ee
	and if $\xi_2 > 0$, then
	\Be \label{xi1=0}
	m(0,\xi_2) = \begin{cases} 
		-\frac{1}{b}\log (c_+/c_-)&\text{ if } c_- >0
		\\
		-\frac{1}{b} \log (-c_+/c_-) -\frac{1}{b}\pi i
		&\text{ if } c_- <0.
	\end{cases}
	\Ee
\end{subequations}
We shall need the following  H\"older continuity condition at the axes.
\begin{lem} \label{Hoelderlemma}
	There is   $C_\circ=C_\circ(b, c_\pm)\ge 1$  such that we have the estimates
	\begin{subequations}
		\begin{align} \label{eq:Hoelderbound-xi2}
		|m(\xi_1,\xi_2)-m (\xi_1,0)| &\le C_\circ \Big(\frac {|\xi_2|}{|\xi_1|^b} \Big)^{\frac{1}{2b}}\,,
		\\
		\label{eq:Hoelderbound-xi1}
		|m(\xi_1,\xi_2)-m (0, \xi_2)| &\le C_\circ \Big(\frac {|\xi_1|^b} {|\xi_2|}
		\Big)^{\frac{1}{2b}}\,.
		\end{align}
	\end{subequations}
\end{lem}

\begin{proof}[Proof of Lemma \ref{Hoelderlemma}.]
	We have $|m(\xi_1,\xi_2)|\le C_{\circ}(b, c_\pm)$ and therefore 
	it suffices to show that \eqref {eq:Hoelderbound-xi2}
	holds for $|\xi_2| \ll |\xi_1|^b$ and 
	\eqref{eq:Hoelderbound-xi1} holds for $|\xi_1|^b\ll |\xi_2|$.
	
	For the proof of \eqref{eq:Hoelderbound-xi2} 
	it suffices to check, by homogeneity and boundedness of $m$,
	\Be \label{eq:Hoelderbound-xi2-one} 
	|m(\pm 1, \xi_2)- m(\pm 1,0)|\lc |\xi_2|^{\beta}, \quad |\xi_2|\le 1, \Ee for some  $\beta\ge (2b)^{-1}$.
	Let \Be \label{ChoiceofA} A= A(\eta)= \frac 12 |\eta|^{- \frac{1}{b+1}}.\Ee  We  have 
	\[
	m(1, \xi_2)- m(1,0)= \sum_{j=1}^3(I_{j,+}(c_+b\xi_2) - I_{j,-}(c_-b\xi_2)) \] where 
	\begin{align*} 
	I_{1,\pm}(\eta) 
	&= \int_0^{A(\eta)}
	e^{\mp it}(e^{-i t^b\eta/b}-1) \frac{dt} t\,,
	\\
	I_{2,\pm} (\eta) &=\int_{A(\eta)}^\infty
	e^{\mp it-i t^b\eta/b} \frac{dt} t\,,
	\\
	I_{3,\pm}(\eta) &= -\int_{A(\eta)}^\infty
	e^{\mp it} \frac{dt} t\,.
	\end{align*}
	
	Clearly $$|I_{1,\pm}(\eta)  |\le
	\int_0^A t^{b-1} |\eta| b^{-1} dt = A^b b^{-2}|\eta|.$$ 
	By integration by parts, 
	$$|I_{3,\pm}| \le  2 A^{-1}.$$
	By our choice \eqref{ChoiceofA} 
	$$|I_{1,\pm}(\eta) |+|I_{3,\pm}(\eta) | \lc |\eta|^{\frac{1}{b+1}}
	$$

	We may assume $|\eta|<1$. Let $B_1=B_1(\eta)= |\eta^{-1/(b-1)}|/2$ and $B_2=B_2(\eta)= 2|\eta^{-1/(b-1)}|$.
	Then $B_1(\eta)\ge A(\eta)$ and we split
	$$ 
	I_{2,\pm} (\eta) =\int_{A}^{B_1} + \int_{B_1}^{B_2} + \int_{B_2}^\infty 
	e^{i\psi(t) } t^{-1} dt
	$$
	with $\psi(t) =\mp t- t^b\eta/b$.

	Note that for $|t|\le B_1$ we have $1/2<|\psi'(t)|\le 2$ and thus, by van der Corput's lemma with first derivative we have $|\int_A^{B_1} (...) dt| \lc A^{-1}$.
	
	Note  that  $|\psi''(t)|=|\eta|(b-1) t^{b-2}$. 
	For the second integral we apply  van der Corput's lemma with second derivatives and get
	$|\int_{B_1}^{B_2} (...) dt| \lc |B_1|^{-1} |\eta|^{-1/2} (b-1)^{-1/2} |B_1|^{-(b-2)/2}\lc (b-1)^{-1/2} |\eta|^{1/(2b-2)}$.
	
	Finally for the third integral we use that 
	$|\psi'(t)|\approx |\eta|t^{b-1} $ and 
	$|\psi''(t) |\approx |\eta| (b-1) t^{b-2}$ and a straightforward  integration by parts argument yields the bound $O(|\eta|^{-1} B_2^{-b}) = 
	O(|\eta|^{\frac{1}{b-1}})$. 
	
	The estimate for 
	$m(-1,\xi_2)-m(-1,0)$ is analogous. 
	Altogether we obtain \eqref{eq:Hoelderbound-xi2-one} with $\beta=
	\min\{(b+1)^{-1}, (2b-2)^{-1}\}$, and we have $\beta\ge (2b)^{-1}$.

	We now turn to the proof of \eqref{eq:Hoelderbound-xi1}.
	It suffices to check, by homogeneity and boundedness of $m$,
	\Be \label{eq:Hoelderbound-xi1-one}
	|m(\xi_1, \pm 1)- m(0, \pm 1)|\lc 
	|\xi_1|^{1/2} , \quad |\xi_1|\le 1. \Ee
	Let \Be \label{ChoiceofB} B= B(\xi_1)= (a|\xi_1|)^{-1/2} \text{ where } a= \min_\pm (bc_\pm/2)^{\frac{2}{b-1}}.\Ee  
	
	We  have 
	\[
	m( \xi_1, 1)- m(0,1)
	= \sum_{j=1}^3\big(II_{j,+}(\xi_1) - II_{j,-}(\xi_1)\big) \] where 
	\begin{align*} 
	II_{1,\pm}(\xi_1) 
	&= \int_0^{B(\xi_1) }
	(e^{\mp it\xi_1}-1)e^{-i c_\pm t^b} \frac{dt} t \,,
	\\
	II_{2,\pm} (\xi_1) &=\int_{B(\xi_1) }^\infty
	e^{\mp it\xi_1-i c_\pm t^b} \frac{dt} t \,,
	\\
	II_{3,\pm}(\xi_1) &= -\int_{B(\xi_1) }^\infty
	e^{\mp it} \frac{dt} t\,.
	\end{align*}
	
	The estimation of these terms is  straightforward; we get
	\[ 
	|II_{1,\pm}(\xi_1)|\lc |\xi_1| B(\xi_1) 
	\] and 
	\[|II_{3,\pm}(\xi_1)|  \lc B(\xi_1)^{-1}
	\]
	and both terms are $O(|\xi|^{1/2})$, by our choice \eqref{ChoiceofB}. By this choice we also have 
	$2 \le |c_\pm|b t^{b-1}$ for $t\ge B(\xi_1)$
	which implies that for $|\xi_1|\le 1$
	$$ \frac 12|c_{\pm}| bt^{b-1} \le 
	|\partial_t (\mp t\xi_1- c_\pm t^{b})|  
	\le 2|c_{\pm}| bt^{b-1} \text{ for } t\ge B(\xi_1).
	$$
	Integration by parts now shows that
	\[
	|II_{2,\pm}(\xi_1)|  
	\lc  B(\xi_1)^{-b} 
	\]
	which is $O(|\xi_1|^{b/2})$, hence also $O(|\xi_1|^{1/2})$.
	The term 
	$m( \xi_1, -1)- m(0,-1)$ is similarly estimated.
	This completes the proof of 
	\eqref{eq:Hoelderbound-xi1-one}.
\end{proof} 

\subsection{\it Reduction to a lower bound for a lacunary maximal operator}
Recall that $\tilde U\subset U$ with 
$\fN(\tilde U)<\infty$.
Let $\fJ$ be the collection of all integers $n$ such that
$[2^n, 2^{n+1}]$ has  nonempty intersection with $\tilde U$, thus $\fN(\tilde U)=1+\#\fI$. Let
\Be \label{KUdef} K=K(\tilde U)= (C_\circ \fN(\tilde U))^{2b}
\Ee
where $C_\circ$ is as in
\eqref{eq:Hoelderbound-xi2}, \eqref{eq:Hoelderbound-xi1}.
Let $\fI'$ be a {\it maximal} subfamily of $\fI$ with the condition
\Be\label{separationinfI'} n_1\in \fI', \,\,n_2\in \fI', \,\, n_1<n_2 \,\, \implies n_2-n_1+1\ge \log_2 (8 K^2).
\Ee

Pick an integer $M$ such that $M+1$ is of the form $2^\mu$ with $\mu\in \bbN$ and such that 
$$
\frac {\fN(\tilde U)}{\log_2 (16 K^2) } =
\frac {\fN(\tilde U)}{4+4b \log_2 (C_\circ \fN(\tilde U))} \in [M,2M).
$$
We may assume that the displayed quantity is $\ge e^{100}$, so that the logarithm of this quantity is comparable to $\log M$ (otherwise the desired lower bound
for $\|\cH^U\|_{L^2\to L^2} $ just follows from the 
trivial lower bound for the  Hilbert transform along a fixed curve).

We may now pick an increasing  sequence  $\{u_j\}_{j=1}^M$ such that each $u_j$ belongs to $\tilde U$ and to  exactly one interval determined by the collection $\fI'$. Hence  we have
\Be\label{lacunarity} \frac{u_{j+1}}{u_j}\ge 16 K^2 \,.\Ee

Given the reduction in \S\ref{p=2redux} the  lower bound $\sqrt{\log (\fN(U))} $ in Theorem~\ref{lowerboundtheorem}
follows from
\begin{prop}\label{kara-para}
	Let $\tilde U$ and $\{u_j\}_{j=1}^M$ be as above.  Then there is  $c>0$ such that
	\[
	\sup_{\|f\|_2=1} \big \|\sup_{1 \leq j \leq M} |\H^{(u_j)}f| \big \|_{2} \ge c \sqrt{\log M} \,.
	\]
\end{prop}
The proof of this proposition is 
based on a construction by Karagulyan \cite{Kar07}.

%
%

\subsection
{\it A theorem of Karagulyan}
We will invoke the following proposition, which is a small generalization of the main theorem of Karagulyan \cite{Kar07} (see also \cite{LMP17}). 
For ${\mu} \in \N$, let $$W_{\mu} = \{\emptyset\} \cup \bigcup_{\ell=1}^{{\mu}-1} \{0,1\}^{\ell}$$ be the set of binary words of length at most ${\mu}-1$, and let \[\tau \colon W_{\mu} \to \{1, \dots, 2^{\mu}-1\}\] be the bijection given by $\tau (\emptyset)=2^{\mu-1}$ and 
$$
\tau(w) = w_1 2^{\mu-1} + w_2 2^{\mu-2} + \dots + w_{\ell} 2^{\mu-\ell} + 2^{\mu-\ell-1}
$$
if $w = w_1 w_2 \dots w_{\ell}$ for some $\ell\in \{1,\dots,\mu-1\}$, and each $w_1, \dots, w_{\ell} \in \{0,1\}$. Observe that for a word $w$ of length $\ell$, $\tau(w)$ is divisible by $2^{\mu-\ell-1}$ but not by $2^{\mu-\ell}$.

\begin{prop} \label{prop:kara} 
	Let ${\mu}$ be any positive integer, $M = 2^{\mu} - 1$, and let $S_1, \dots, S_M$ be pairwise disjoint subsets of the (frequency) plane $\R^2$, so that every $S_j$ contains balls of arbitrarily large radii (in other words, for every $1 \leq j \leq M$ and every $R > 0$, $S_j$ contains some ball of radius~$R$). Then there exists an $L^2$ function $f$ on $\R^2$, that admits an orthogonal decomposition 
	\[
	f = \sum_{w \in W_{\mu}} f_w,
	\]
	where 
	\beq \label{eq:ortho}
	\supp \widehat{f_w} \subset S_{\tau(w)} \quad \text{for all $w \in W_{\mu}$, and}
	\endeq
	\beq \label{eq:fwLpf}
	\|f\|_{L^2}^2 = \sum_{w \in W_{\mu}} \|f_w\|_{L^2}^2 \le 2;
	\endeq
	in addition,
	\beq \label{eq:sqrootlower}
	\Big \| \sup_{1 \leq j \leq M}  \Big| \sum_{w \in W_{\mu} \colon \tau(w) \geq j} f_w \Big| \Big \|_{L^2} \ge \frac{\sqrt{\mu} }{100} \|f\|_{L^2}.
	\endeq
\end{prop}

Accepting this for the moment, we prove Proposition~\ref{kara-para}.

\subsection{\it Proof of Proposition~\ref{kara-para}}
As before, suppose $c_+ > 0$. Let 
\[\rho=\begin{cases} - \frac{1}{b}\log(c_+/c_-) &\text{ if $c_->0$}, 
\\
-\frac{1}{b}\log (-c_+/c_-) -\frac{1}{b}\pi i 
&\text{ if $c_-<0$}. 
\end{cases}
\]
Then $m(0,\xi_2)=\rho$ for $\xi_2>0$ and $m(\xi_1,0)= -\pi i$ for $\xi_1>0$ (cf. \eqref{xi1=0}, \eqref{xi2=0}).
Let $K$ as in \eqref{KUdef}, then $$C_\circ K^{-\frac 1{2b}} \le (\fN(\tilde U))^{-1}\le M^{-1}.$$ From \eqref{eq:Hoelderbound-xi2} and 
\eqref{eq:Hoelderbound-xi1}  we see, for $\xi_1>0$, $\xi_2>0$ 

\begin{subequations}
	\begin{alignat}{2} 
	\label{boundaryestxi2=0}
	&\xi_2/\xi_1^b\le K^{-1} \,& &\implies\, |m(\xi_1,\xi_2)+\pi i| \le C_\circ K^{-\frac{1}{2b}}\le M^{-1}.
	\\
	\label{boundaryestxi1=0}
	&\xi_2/\xi_1^b\ge K \,& &\implies\, |m(\xi_1,\xi_2)-\rho| \le C_\circ K^{-\frac{1}{2b}}\le M^{-1}
	\end{alignat}
\end{subequations}

For $1 \leq j \leq M$, define 
\Be\label{Sjdef} 
S_j=\big\{ (\xi_1,\xi_2): \xi_1>0,\,\xi_2>0,\, \,
\frac{1}{2Ku_j}  <
\frac{ \xi_2}{\xi_1^b}  < \frac{1}{ Ku_j} \big\},
\Ee
so that the $S_j$ are pairwise disjoint, and contain balls of arbitrarily large radii. By Proposition~\ref{prop:kara}, there exists an $L^2$ function $f = \sum_{w \in W_{\mu}} f_w$ on $\R^2$, such that (\ref{eq:ortho}), (\ref{eq:fwLpf}) and (\ref{eq:sqrootlower}) hold. Now for $1 \leq j \leq M$, 
\begin{multline*}
|\H^{(u_j)} f(x)-\rho f(x) | 
\geq \Big | \sum_{{\substack{w \in W_{\mu} \colon \\\tau(w) \geq j}}} (\pi i+\rho)  \, f_w(x) \Big | 
\\
\quad - \Big | \sum_{\substack{w \in W_{\mu} \colon \\ \tau(w) \geq j}} \big( \H^{(u_j)}f_w(x) + \pi i  f_w(x) \big) \Big | - \Big | \sum_{{\substack{w \in W_{\mu} \colon 
			\\ \tau(w) < j}}} \big( \cH^{(u_j)}f_w(x) -\rho f_w(x)\big) \Big |,
\end{multline*}
and thus, with $c_0 = \pi(1-\frac{1}{b})$, 
\begin{multline}\label{mainpartwithtwoerrors}
\sup_{1\le j\le M}
|\H^{(u_j)} f(x)-\rho f(x) | 
\geq c_0 \sup_{1\le j\le M}\Big | \sum_{{\substack{w \in W_{\mu} \colon \\\tau(w) \geq j}}}  f_w(x) \Big | 
\\
\quad - \sup_{1\le j\le M}
\Big | \sum_{\substack{w \in W_{\mu} \colon \\ \tau(w) \geq j}} ( \H^{(u_j)}+\pi i)f_w(x)  \Big | - \sup_{1\le j\le M} \Big | \sum_{{\substack{w \in W_{\mu} \colon 
			\\ \tau(w) < j}}} ( \cH^{(u_j)}-\rho )f_w(x) \big) 
\Big|.
\end{multline}

Now $\supp \widehat {f_w} \in S_{\tau(w)}$. If $\tau(w)\ge j$, then for $\xi\in \supp \widehat {f_w }$, we have $u_j\xi_2/\xi_1^b< u_{\tau(w)} \xi_2/\xi_1^b<K^{-1}$ and therefore, by \eqref{boundaryestxi2=0}, we have
$|m(\xi_1, u_j\xi_2)+\pi i| \le M^{-1}$ for $\xi\in \supp \widehat {f_w }$. Hence

\Be\label{tauwgejerror}
\big\| (\H^{(u_j)} + \pi i) f_w \big \|_{2} \leq M^{-1}  \|f_w\|_{2} \quad \text{if $\tau(w) \geq j$}.\Ee

Moreover if $\tau(w)<j$ we have, for $\xi\in \supp \widehat {f_w }$,
\[
u_j \frac{\xi_2}{\xi_1^b} 
= \frac{u_j}{u_{\tau(w)}}u_{\tau(w) }\frac{\xi_2}{\xi_1^b} \ge  16 K^2 \frac 1{2K}= 8K
\]
and hence, by \eqref{boundaryestxi1=0},
$|m(\xi_1, u_j \xi_2)-\rho|\le M^{-1}$
for $\xi\in \supp \widehat {f_w }$.  Thus 
\Be
\label{tauw<jerror}
\big \| (\H^{(u_j)} -\rho) f_w \big \|_{2} \leq M^{-1} \|f_w\|_{2} \quad \text{if $\tau(w) < j$}.
\Ee
Statements \eqref{tauwgejerror}
and \eqref{tauw<jerror} 
imply 
\begin{align}
\label{sumtauwgejerror}
&\Big\| \sup_{1 \leq j \leq M} \Big | \sum_{\substack{w \in W_{\mu} \colon \\\tau(w) < j}} (\H^{(u_j)}-\rho) f_w \Big | \Big \|_{2}  \lc \|f\|_2
\\
\label{sumtauw<jerror} 
&\Big \| \sup_{1 \leq j \leq M} \Big | 
\sum_{\substack{w \in W_{\mu} \colon \\ \tau(w) \geq j}} ( \H^{(u_j)} + \pi i ) f_w \Big | \Big \|_{2} \lesssim \|f\|_{2}.
\end{align} 
Indeed, to obtain \eqref{sumtauw<jerror} we use 
the Cauchy-Schwarz inequality  in the $w$ sum and replace a sup in $j$ by an $\ell^2$ norm, then interchange integrals and sums and apply  \eqref{tauw<jerror} to get
\[
\begin{split}
&\Big\| \sup_{1 \leq j \leq M} \Big | \sum_{\substack{w \in W_{\mu} \colon \\\tau(w) < j}} (\H^{(u_j)}-\rho) f_w \Big | \Big \|_{2} 
\\ &\leq  M^{1/2} \Big \| \Big( \sum_{j=1}^M \sum_{\tau(w) < j} \ | (\H^{(u_j)}-\rho)f_w  |^2 \Big)^{1/2} \Big \|_{2} \\
&= M^{1/2} \Big( \sum_{j=1}^M \sum_{\tau(w) < j} \big \| (\H^{(u_j)}-\rho) f_w \big \|_{2}^2 \Big)^{1/2}\\
&\leq  M^{1/2} \Big( \sum_{j=1}^M M^{-2}  \sum_{w }  \| f_w  \|_{2}^2 \Big)^{1/2} 
\lesssim  \|f\|_{2}
\end{split}
\]
(the last line following from (\ref{eq:fwLpf})).
Inequality \eqref{sumtauwgejerror} is proved in exactly the same way (relying on 
\eqref{tauwgejerror}).

Now we go back to \eqref{mainpartwithtwoerrors}, use 
(\ref{eq:sqrootlower}) for the main part and 
\eqref{sumtauwgejerror}, \eqref{sumtauw<jerror} for the two error terms. Then we get  
\[
\Big \| \sup_{1 \leq j \leq M} |(\H^{(u_j)}-\rho) f| \Big \|_{2} \geq c \sqrt{\mu} \|f\|_{2}
\]
for some  constant $c= c(b, c_\pm)> 0$. 
If $\sqrt{\mu} \ge 2|\rho| / c$ this also implies 
\[
\Big \| \sup_{1 \leq j \leq M} |\H^{(u_j)} f| \Big \|_{2} \geq (c/2)  \sqrt{\mu} \|f\|_{2}.
\]
This completes the proof
of Proposition~\ref{kara-para}, except for  
Proposition~\ref{prop:kara}. \qed

\subsection{ \it Proof of Proposition~\ref{prop:kara}}

Fix a non-negative Schwartz function $\phi$ on $\R^2$ with $\int_{\R^2} \phi(x) dx = 1$, such that $\widehat{\phi}$ is supported in the unit ball $B(0,1)$ centered at the origin. Define the frequency cutoff $\phi_\rho$ by  $$\phi_{\rho}(x) := \rho^2 \phi(\rho x).$$ Then $\widehat{\phi_\rho}$ is supported on $B(0,\rho)$. 

The following lemma explains what we actually construct, in order to prove Proposition~\ref{prop:kara}:

\begin{lem} \label{lem:kara_const} 
	Let ${\mu} \in \N$, $M = 2^{\mu} - 1$, and let $S_1, \dots, S_M$ be as given in Proposition~\ref{prop:kara}. Then there exist a sequence of sets $\{E_w\}_{w \in W_{\mu}}$, modulation frequencies $\{\xi_w\}_{w \in W_{\mu}} \subset \R^2$, and 
radii $\{\rho_w\}_{w\in W_\mu}$	 
	 such that the following holds:
	\begin{enumerate}[(a)]
		\item For every $w \in W_{\mu}$, $E_w \subset [0,1]^2$, and for every $w \in W_{\mu-1}$, $E_w$ is the disjoint union of $E_{w0}$ and $E_{w1}$ 
Also, $E_{\emptyset} = [0,1]^2$. For $\ell=0,\dots, \mu-1$, $[0,1]^2$ is a disjoint union of the $E_w$ with $\text{length}(w)=\ell$, and 
		\beq \label{eq:sumEw}
		\sum_{w \in W_{\mu}} \bbone_{E_w}(x) = \mu.
		\endeq
		for every  $x \in [0,1]^2$.
		\item For every $w \in W_{\mu}$, 
		\begin{align} \label{eq:lwchoice}
		\|\bbone_{E_w}*\phi_{\rho_w} - \bbone_{E_w} \|_{L^2} &\leq 2^{-\mu-10}, \\ \label{eq:coslowerbdd}
		\int_{E_w} |\cos(\inn{\xi_w}{ x})| dx &\geq \frac{|E_w|}{3}, 
		\end{align} 
		\Be
		\label{eq:balllocation}
		B(\xi_w,\rho_w) \subset S_{\tau(w)}.
		\Ee
		\item For every $w \in W_{\mu-1}$, we have
		\Be \label{eq:cossign}
		\begin{aligned}
		\cos(\inn{\xi_w}{x}) &\ge  0 \quad \text{if $x \in E_{w0}$}, \\
		\cos(\inn{\xi_w}{x}) &<  0 \quad \text{if $x \in E_{w1}$}.
		\end{aligned}
		\Ee
	\end{enumerate}
\end{lem}

With this lemma we can prove Proposition~\ref{prop:kara} as follows.

\begin{proof}[Proof of Proposition~\ref{prop:kara}]
	For every $w \in W_{\mu}$, let $E_w$, $\rho_w$ and $\xi_w$ be as in Lemma~\ref{lem:kara_const}. We set
	\begin{subequations}
	\beq \label{eq:fwdef}
	f_w(x) := \mu^{-1/2} e^{i \inn{\xi_w}{x}} \bbone_{E_w}*\phi_{\rho_w}(x),
	\endeq
	and let \Be\label{f=sumfw}f := \sum_{w \in W_{\mu}} f_w.\Ee
	\end{subequations}
	Then the support of $\widehat{f_w}$ is contained inside $B(\xi_w,\rho_w)$, so (\ref{eq:ortho}) follows from (\ref{eq:balllocation}). 
	Also, the $\widehat{f_w}$'s are supported in the sets $S_{\tau(w)}$ which are disjoint and thus  by orthogonality we have
	\[
	\|f\|_{2} = \Big \| \Big( \sum_{w \in W_{\mu}} |f_w|^2 \Big)^{1/2} \Big\|_{2}.
	\]
	But, from (\ref{eq:lwchoice}), we have
	\beq \label{eq:fwEwapprox}
	\Big \| f_w - \mu^{-1/2} e^{i \inn{\xi_w}{x}} \bbone_{E_w} \Big \|_{2} \leq 2^{-\mu-10}.
	\endeq
	Observe
	\begin{multline*}
	\Big( \sum_{w \in W_{\mu}} |f_w|^2 \Big)^{1/2} 
	\\ \quad\leq  \Big( \sum_{w \in W_{\mu}} \Big|f_w - \mu^{-1/2} e^{i \inn{\xi_w}{x}} \bbone_{E_w} \Big|^2 \Big)^{1/2} + \Big( \sum_{w \in W_{\mu}} \Big|\mu^{-1/2} e^{i \inn{\xi_w}{x}} \bbone_{E_w} \Big|^2 \Big)^{1/2},
	\end{multline*}
	and  using (\ref{eq:sumEw}) to simplify the second term we get
	\[
	\Big( \sum_{w \in W_{\mu}} |f_w|^2 \Big)^{1/2} 
	\leq \Big(\sum_{w \in W_{\mu}} \Big|f_w - \mu^{-1/2} e^{i \inn{\xi_w}{x}} \bbone_{E_w} \Big|^2\Big)^{1/2}  + \bbone_{[0,1]^2}
	\]
	for almost every $x \in \R^2$. Taking $L^2$ norms of both sides, and using (\ref{eq:fwEwapprox}), we have 
	\[
	\Big \| \Big( \sum_{w \in W_{\mu}} |f_w|^2 \Big)^{1/2} \Big\|_{2} \leq 2^{-5-\mu/2}+1 <2.
	\]
	Thus (\ref{eq:fwLpf}) follows. 
	
	Lastly we have to verify (\ref{eq:sqrootlower}). To do so, we first introduce an auxiliary family of functions $\{F_w\}_{w \in W_{\mu}}$, where
	\beq \label{eq:Fwdef}
	F_w := \re f_w \,\bbone_{E_w}.
	\endeq
	These $F_w$'s satisfy three key properties, namely 
	\beq \label{eq:Fwfwapproxsum}
	\sum_{w \in W_{\mu}} \|F_w - \re f_w\|_{L^2} \le 2^{-10},
	\endeq
	\beq \label{eq:treeest}
	\frac 13 \,\le \,
	\frac{ 
	\sup_{1 \leq j \leq M} \Big| \sum_{w \in W_{\mu} \colon \tau(w) \geq j} F_w(x) \Big | }
	{ \sum_{w \in W_{\mu}} |F_w(x)| } \,\le 1 \quad \text{for a.e. $x \in [0,1]^2$},
	\endeq
	and
	\beq \label{eq:L1Linftyest}
	\frac{\sqrt\mu}{4}\le \Big \| \sum_{w \in W_{\mu}} |F_w| \Big\|_{1} \le 
	\Big \| \sum_{w \in W_{\mu}} |F_w| \Big\|_{2} \le 
	\Big \| \sum_{w \in W_{\mu}} |F_w| \Big \|_{{\infty}} \le 
	\sqrt{\mu}.
	\endeq
	Indeed, (\ref{eq:Fwfwapproxsum}) will be a consequence of  
	\beq \label{eq:Fwfwapprox}
	\|F_w - \re f_w\|_{L^2} \lesssim 2^{-\mu-10} \quad \text{for all $w \in W_{\mu}$}.
	\endeq 
	Since $F_w - \re f_w = \re f_w \bbone_{\R^2 \setminus E_w}$, heuristically, (\ref{eq:Fwfwapprox}) says that the real part of each $f_w$ is essentially supported on $E_w$: the $L^2$ norm of $\re f_w$ outside $E_w$ is small. Furthermore, (\ref{eq:treeest}) says that there isn't much cancellation, if we first order the $F_w$'s according to the value of $\tau(w)$, and then sum successively; this will be achieved by showing that $\{F_w\}_{w \in W_{\mu}}$ form a \emph{tree system} in the sense of Karagulyan \cite{Kar07} (who credits the idea to Niki\v sin and Ul'janov \cite{NikUly}). 
	
	Let us now establish the three key properties of the $F_w$'s, namely (\ref{eq:Fwfwapproxsum}), (\ref{eq:treeest}) and (\ref{eq:L1Linftyest}). Since $F_w - \re f_w = \re f_w \bbone_{(E_w)^\complement}$, and since
	\beq \label{eq:refw}
	\re f_w(x) = \frac{1}{\sqrt{\mu}} \cos(\inn{\xi_w}{x}) \bbone_{E_w}*\phi_{\ell_w}(x),
	\endeq
	we have
	\begin{align*}
	\|F_w - \re f_w\|_{L^2(\R^2)} &= \big\| \mu^{-1/2} \cos(\inn{\xi_w}{x}) \bbone_{E_w}*\phi_{\ell_w} \big\|_{L^2(\R^2 \setminus E_w)} 
	\\
	&\leq \big\| \bbone_{E_w} - \bbone_{E_w}*\phi_{\ell_w} \big\|_{L^2(\R^2\setminus E_w)} \le 2^{-\mu-10}
	\end{align*}
	by (\ref{eq:lwchoice}). This establishes (\ref{eq:Fwfwapprox}), and (\ref{eq:Fwfwapproxsum}) follows by summing over $w \in W_{\mu}$. 
	
	Next we verify \eqref{eq:treeest}. 
	The second inequality in \eqref{eq:treeest} is immediate by the triangle inequality. For the first, we observe from (\ref{eq:refw}) that if $x \in E_w$, then $F_w(x)$ has the same sign as $\cos(\inn{\xi_w}{x})$ since $\bbone_{E_w}*\phi_{\ell_w}$ is everywhere positive. We claim that for almost every $x \in [0,1]^2$, there exists $j=j(x)$ such that $F_w(x) \geq 0$ for every $w \in W_{\mu}$ with $\tau(w) \geq j$, and $F_w(x) < 0$ for every $w \in W_{\mu}$ with $\tau(w) < j$. This is because for almost every $x \in [0,1]^2$, there exists a unique word $w(x) = w_1 \dots w_{\mu-1}$ of length $\mu-1$ such that $x \in E_{w(x)}$. By (\ref{eq:cossign}), it follows that, for every $\ell = 0, 1, \dots, \mu-2$,
	\begin{align*} 
	F_{w_1 \dots w_{\ell}}(x) 
	&	> 0 \quad \text{if $w_{\ell+1} = 0$}, \\
	F_{w_1 \dots w_{\ell}}(x) 
	&< 0 \quad \text{if $w_{\ell+1} = 1$},
	\end{align*}
	and that $F_{w'}(x) = 0$ if $w' \in W_{\mu} \setminus \{\emptyset, w_1, w_1 w_2, \dots, w_1 \cdots w_{\mu-1}\}$. But  
	$$
	\tau(w_1 \dots w_{\ell}) = w_1 2^{\mu-1} + \dots + w_{\ell} 2^{\mu-\ell} + 2^{\mu-\ell-1},
	$$
	while 
	$$
	\tau(w(x)) = w_1 2^{\mu-1} + \dots + w_{\ell} 2^{\mu-\ell} + w_{\ell+1} 2^{\mu-\ell-1} + \dots + w_{\mu-1} 2^1 + 2^0.
	$$
	This shows that for every $\ell = 0, 1, \dots, \mu-2$,
	\begin{align*}
	\tau(w_1 \dots w_{\ell}) 
	&	> \tau(w(x)) \quad \text{if $w_{\ell+1} = 0$}, 
	\\
	\tau(w_1 \dots w_{\ell}) &< \tau(w(x)) \quad \text{if $w_{\ell+1} = 1$}.
	\end{align*}
	Thus for any $w' \in W_{\mu}$, one has
	\begin{align*}
	F_{w'}(x) 
	&\geq 0 \quad \text{if $\tau(w') > \tau(w(x))$}, \\
	F_{w'}(x) &\leq 0 \quad \text{if $\tau(w') < \tau(w(x))$}.
	\end{align*}
	If $F_{w(x)}(x) \ge  0$, we set $j(x) = \tau(w(x))$; if $F_{w(x)}(x) < 0$, we set $j(x) = \tau(w(x)) + 1$. It follows that that $F_w(x) \geq 0$ whenever $\tau(w) \geq j(x)$, and $F_w(x) \leq 0$ whenever $\tau(w) <  j(x)$. 
We distinguish two cases now. In the first case we have 
\[ \Big|  \sum_{w \in W_{\mu} \colon \tau(w) \geq j(x)} F_w(x) \Big| \geq \frac{1}{3} \sum_{w \in W_{\mu}} |F_w(x)|.	\]
	In the opposite case, we have 
	$|  \sum_{w \in W_{\mu} \colon \tau(w) \geq j(x)} F_w(x) | 
	< \frac{1}{3} \sum_{w \in W_{\mu}} |F_w(x)|$, so $|  \sum_{w \in W_{\mu} \colon \tau(w) < j(x)} F_w(x) | \geq \frac{2}{3} \sum_{w \in W_{\mu}} |F_w(x)|$.  Then
\begin{align*}  &\Big|  \sum_{w \in W_{\mu}} F_w(x) \Big| \geq 
	\Big|  \sum_{\substack{w \in W_{\mu} \colon\\ \tau(w)<j(x)}} F_w(x) \Big| 
	- \Big|  \sum_{\substack{w \in W_{\mu} \colon \\\tau(w) \geq j(x)}} F_w(x)\Big |
	\\
	 &\ge 
	  \sum_{\substack{w \in W_{\mu} \colon\\ \tau(w)<j(x)}} |F_w(x)| 
	- \frac{1}{3} \sum_{w \in W_{\mu}} |F_w(x)|
	\ge 
	\frac{1}{3} \sum_{w \in W_{\mu}} |F_w(x)|.	
	\end{align*}
	Hence in both cases 
	$$
	\sup_{1 \leq j \leq M} \Big| \sum_{w \in W_{\mu} \colon \tau(w) \geq j} F_w(x) \Big| \geq \frac{1}{3} \sum_{w \in W_{\mu}} |F_w(x)|
	$$
	for every $x \in [0,1]^2$. 
	This completes the proof of (\ref{eq:treeest}).
	
	Finally, we have to verify (\ref{eq:L1Linftyest}). Note that $F_w$ is supported on $[0,1]^2$ for every $w \in W_{\mu}$, and for almost every $x \in [0,1]^2$, there exists at most $\mu$ words $w \in W_{\mu}$ for which $F_w(x) \ne 0$. Furthermore, $|F_w(x)| \leq \mu^{-1/2}$ for every $x \in [0,1]^2$ and every $w \in W_{\mu}$. Thus, we have
	\[
	\Big \| \sum_{w \in W_{\mu}} |F_w| \Big \|_{1} \leq
	\Big \| \sum_{w \in W_{\mu}} |F_w| \Big \|_{2} \leq
	\Big \| \sum_{w \in W_{\mu}} |F_w| \Big\|_{{\infty}} \leq \sqrt{\mu}.
	\] 
	Next, for the lower bound,
	\[
	\Big \| \sum_{w \in W_{\mu}} |F_w| \Big \|_{1} 
	= \sum_{w \in W_{\mu}} \int_{E_w} \mu^{-1/2}|\cos(\inn{\xi_w }{x}) \bbone_{E_w}*\phi_{\ell_w}(x)| dx 
	\]
	which is
	\[
	\begin{split}
	&\ge\frac{1}{\sqrt\mu}\sum_{w \in W_{\mu}} 
	 \int_{E_w} \Big( |\cos(\inn{\xi_w}{x})|  -  
	 \big|\cos(\inn{\xi_w}{x})[\bbone_{E_w} - \bbone_{E_w}*\phi_{\ell_w}]\big|\Big)\,dx \\
	&\geq \frac{1}{\sqrt\mu}\sum_{w \in W_{\mu}} \Big(\int_{E_w}  |\cos(\inn{\xi_w}{x})| dx - \|\bbone_{E_w} - \bbone_{E_w}*\phi_{\ell_w}\|_{L^2}
	|E_w|^{1/2}\Big)\\
	&\ge \frac{1}{\sqrt{\mu}} \sum_{w \in W_{\mu}}\Big(\frac{ |E_w|}{3}  - 2^{-\mu-10} \Big)
	\ge \frac{\sqrt{\mu}}{3} - 2^{-\mu-10}\sqrt{\mu}\ge 
	\frac{\sqrt\mu}{4},
	\end{split}
	\]
	where for the last line  we have  used  (\ref{eq:coslowerbdd}), (\ref{eq:lwchoice}) and (\ref{eq:sumEw}).
	This completes the proof of (\ref{eq:L1Linftyest}).
	
	We will now return to the proof of (\ref{eq:sqrootlower}). First,
	\[
	\begin{split}
	\sup_{1 \leq j \leq M} \Big | \sum_{w \in W_{\mu} \colon \tau(w) \geq j} f_w(x) \Big | 
	\geq \sup_{1 \leq j \leq M} \Big | \sum_{w \in W_{\mu} \colon \tau(w) \geq j} \re f_w(x) \Big | &\\
	\geq \sup_{1 \leq j \leq M} \Big | \sum_{w \in W_{\mu} \colon \tau(w) \geq j} F_w(x) \Big | - \sum_{w \in W_{\mu}} |F_w(x) - \re f_w(x)|&,
	\end{split}
	\]
	which by (\ref{eq:treeest}) is 
	$$
	\geq \frac 13 \sum_{w \in W_{\mu}} |F_w(x)| - \sum_{w \in W_{\mu}} |F_w(x) - \re f_w(x)|.
	$$
	From (\ref{eq:Fwfwapproxsum}) and (\ref{eq:L1Linftyest}), we then have
	$$
	\Big \| \sup_{1 \leq j \leq M} \Big | \sum_{w \in W_{\mu} \colon \tau(w) \geq j} f_w \Big | \Big \|_{L^2} \ge\frac{ \sqrt{\mu}}{12}- 2^{-10} \ge \frac{\sqrt{\mu}}{50}.
	$$
	Hence (\ref{eq:sqrootlower}) follows from (\ref{eq:fwLpf}). This finishes the proof of Proposition  \ref{prop:kara}, except for the proof of 
	Lemma~\ref{lem:kara_const}.
\end{proof}

The proof of Lemma \ref{lem:kara_const} 
is done by induction over the length of words.
The basic step is contained in 
\begin{lem}\label{basicstep}
Given $\eps>0$,  a set $E$ of finite measure and 
a set $S$ in frequency space that contains balls of arbitrary large radii, there exist $\rho_0>0$, a frequency $\xi_0$ and a ball $B= B(\xi_0,\rho_0) \subset S$ such that
 $\|\phi_{\rho_0}*\bbone_E-\bbone_E\|_2<\eps$ and 
$\int_E|\cos(\inn{\xi_0} x)|\, dx\ge |E|/3$.
\end{lem}

\begin{proof}
Since $\{\phi_\rho\}_{\rho>0}$ form an approximation of the identity there is $R_1=R_1(S,E,\eps)$ such that
\Be \label{apprixid} \|\phi_\rho*\bbone_E-\bbone_E\|_2<\eps 
\Ee 	
 for  $\rho>R_1$.
Also observe that 
 \[
	\begin{split}
	\liminf_{|\xi| \to +\infty} \int_{E} |\cos(\inn{\xi}{x})| dx 
	&\geq \liminf_{|\xi| \to +\infty} \int_{E} \cos^2(\inn{\xi}{x}) dx \\
	&= \lim_{|\xi| \to +\infty} \int_{E} \frac{1+ \cos(2 \inn{\xi}{x})}{2} dx = \frac{|E|}{2},
	\end{split}
	\]
	by the Riemann-Lebesgue lemma. 
	Hence we find $R_2=R_2(S,E,\eps)$ such that
	\Be \label{RLlemma} \int_{E}|\cos(\inn\xi x)|dx\ge |E|/3, 
	\Ee
 for $|\xi|\ge R_2$.
 
By assumption on $S$ we can find a ball $B_0$ of radius $R_0> 
10 \max \{R_1, R_2\}$, centered at some $\Xi_0$  such that $B_0\subset S$. There is a point
$\xi_0\in B(\Xi_0, R_0/2) $  that satisfies $|\xi_0|\ge R_0/4$. Set $\rho_0=R_0/4$. The ball 
$B(\xi_0,\rho_0)$ is contained in $B_0$ and thus in $S$.
Also since $\rho_0\ge R_1$ we have  \eqref{apprixid} 
for $\rho=\rho_0$ 
and since $|\xi_0|> R_2$  we have  \eqref{RLlemma} for $\xi=\xi_0$.
	\end{proof}

\begin{proof}[Proof of Lemma \ref{lem:kara_const}]
	We will construct a sequence of sets $\{E_w\}$, radii $\rho_w$ and modulation frequencies $\xi_w$ using induction on the length of words. We use  $\eps=2^{-\mu-10}$ in Lemma \ref{basicstep}.
		
	First let $E_{\emptyset} = [0,1]^2$.
	We apply Lemma \ref{basicstep}
	 with
	$E=E_\emptyset$ and $S=S_{\tau(\emptyset)}$.
	We thus find $\xi_\emptyset$, $\rho_\emptyset$ 
	such that 
	\eqref{eq:lwchoice}, \eqref{eq:coslowerbdd},
	\eqref{eq:balllocation} 
	hold for $w=\emptyset$.
		We consider the two words of length one, i.e. $0$ and $1$ and let 
\begin{align*}
	E_{0} &:= \{x \in E_\emptyset
	\colon \cos(\inn{\xi_w }{x}) \ge  0\}
	\\
	E_{1} &:= \{x \in E_\emptyset \colon \cos(\inn{\xi_w}{x}) < 0\}
	\end{align*} 
	so that $E_\emptyset $ is a disjoint union of $E_{0}$ and $E_{1}$, and (\ref{eq:cossign}) holds for $w=\emptyset$. Clearly $[0,1]^2$ is a disjoint union of the $E_w$ with words $w$ of length $1$.

		Suppose $E_w, \rho_w, \xi_w$ are defined for all words of length $\ell<\mu-1$.
		Take any word of length $\ell+1$, of the form
		$w0$ or $w_1$ where $w$ is of length $\ell$, and where $E_w,\rho_w, \xi_w$ satisfy 
	\eqref{eq:lwchoice}, \eqref{eq:coslowerbdd},
	\eqref{eq:balllocation}, and where $[0,1]^2 $ is a disjoint union of the $E_w$ with $\text{length}(w)=\ell$.		
We let 
	\begin{align*} 
	E_{w0} &:= \{x \in E_w \colon \cos(\inn{\xi_w}{x}) \ge  0\}\\
	E_{w1} &:= \{x \in E_w \colon \cos(\inn{\xi_w}{x}) < 0\}
	\end{align*}
so that (\ref{eq:cossign}) holds, $E_w$ is a  disjoint union of $E_{w0}$ and $E_{w1}$, and thus 
$[0,1]^2$ is a disjoint union of all $E_{\overline w}$ where $\overline w $ runs over all words of length $\ell+1$.

We now use Lemma \ref{basicstep}  to find $\rho_{w0}, \xi_{w0}$ so that (\ref{eq:lwchoice}),
(\ref{eq:coslowerbdd}) and (\ref{eq:balllocation}) 
hold for $w0$ in place of $w$.
Then we use 
	 Lemma \ref{basicstep} again  to find $\rho_{w1}, \xi_{w1}$ so that (\ref{eq:lwchoice}),
(\ref{eq:coslowerbdd}) and (\ref{eq:balllocation}) 
hold for $w1$ in place of $w$.

At step $\ell=\mu-1$ this completes our construction 
of  $E_w$, $\rho_w$ and $\xi_w$ for all $w \in W_{\mu}$, and all the properties stated in Lemma~\ref{lem:kara_const} are satisfied at every stage of the construction. Note that the balls 
$B(\xi_w,\rho_w)$, $B(\xi_{\tilde w},\rho_{\tilde w})$  are disjoint for different $w$, $\tilde w$ because these balls belong to the disjoint sets  $S_{\tau(w)}$, $S_{\tau(\tilde w)}$, respectively.

Finally we have by our construction, for 
$\ell=0,\dots, \mu-1$, 
$$\sum_{w: \text{length}(w)=\ell} \bbone_{E_w}=\bbone_{[0,1]^2}, 
$$
and we obtain \eqref{eq:sumEw} by summing in $\ell$.
\end{proof}

\appendix

\section{Proof of Proposition \ref{Cotlarprop}} \label{cotlar-appendix}
The proof is a modification of the argument for   the standard Cotlar inequality regarding truncations of singular integrals, \cf. \cite[\S I.7]{Ste93}.

Let $m_j(\xi) =\eta(2^{-j}\xi)m(\xi)$ and let $a_j(\xi) = m_j(2^j\xi)$. We pick  $0<\eps<\min\{\alpha-d,1\}$. Then by assumption \Be\label{multassu}\sup_{j\in \bbZ}\|a_j\|_{\sL^1_\alpha}\le B<\infty\Ee which implies that   $|\cF^{-1} [a_j](x)|\le CB(1+|x|)^{-d-\eps}$,  and thus, with $K_j= \cF^{-1} [m_j]$,
\begin{align*} 
&|K_j(x)|+ 2^{-j} |\nabla K_j(x)| \le C B 2^{jd} (1+2^j|x|)^{-d-\eps}.
\end{align*}
For Schwartz functions $f$ we have $S f=\sum_{j\in \bbZ} K_j* f$ and $S_nf=\sum_{j\le n} K_j*f$.

\begin{lem}\label{cotlarlemmaapp} 
	Fix $\tx\in \bbR^d$ and $n\in \bbZ$, and  let 
	$g(y)= f(y)\bbone_{B(\tx, 2^{-n})}(y)$ and $h=f-g$.
	Then 
	
	(i) $|S_n g(\tx)|\lc B \, M[f](\tx)$.
	
	(ii) $|S_nh(\tx)-S h(\tx)|\lc B\, M[f](\tx)$.

	(iii) For $|w-\tx|\le 2^{-n-1}$ we have 
	$|Sh(\tx)-Sh(w)|\lc B \, M[f](\tx).$
\end{lem}

\begin{proof} By appropriate normalization of the multiplier we may assume $B=1$.
	
	(i) is immediate since  for $j\le n$ 
	$$|K_j*g(\tx)|\lc 2^{jd}\int_{|\tx-y|\le 2^{-n}}|g(y)| dy \lc 2^{(j-n)d} M[g](\tx)$$
	and the assertion follows since $|g|\le |f|$.
	
	For (ii)  notice that $|S_n h(\tx)-S h(\tx)|\le \sum_{j>n}|K_j* h(\tx)|$.
	For $j>n$ we estimate
	\begin{align*} |K_j*h(\tx)|
	&\lc 2^{-j\eps}\int_{|\tx-y|\ge 2^{-n}} 
	|\tx-y|^{-d-\eps}|h(y)|dy
	\\
	&\lc2^{-j\eps} \sum_{l\ge 0} 2^{(n-l)\eps}\intslash_{B(\tx, 2^{l-n})} |h(y)| dy 
	\end{align*}
	where the slashed integral denotes the average. Thus we get 
	\[\sum_{j\ge n}  |K_j*h(\tx)|\lc M[h](\tx)\] and, since $|h|\le |f|$, the assertion follows.
	
	Concerning (iii) we consider the terms $K_j*h(\tx)-K_j*h(w)$ separately for  $j\le n$ and $j>n$. The term $\sum_{j>n} |K_j*h(\tx)|$ was already dealt with in (ii). 
	Since $|w-\tx|\le 2^{-n-1}$ we  have 
	$|w-y|\approx |\tx-y| $ for $|\tx-y|\ge 2^{-n}$ and thus the previous calculation also yields 
	\[\sum_{j> n}  |K_j*h(w)|\lc M[h](\tx) \lc Mf(\tx).\]
	It remains to consider the terms for $j\le n$.
	In that range we write
	$$K_j*h(\tx)-K_j*h(w)
	= \int_0^1\int\limits_{|\tilde x-y|\ge 2^{-n} }\inn{\tx-w}{\nabla K_j (w+s(\tx-w)-y)} h(y) dy\, ds.
	$$ Since $|w-\tx| \le 2^{-n-1}$ we can replace 
	$|w+s(\tx-w)-y| $ in the integrand  with $|\tx-y|$ and  estimate the displayed expression by $C\sum_{l\ge 0} A_{l,j,n}$ where
	\begin{align*} A_{l,j,n}&= 2^j|\tx-w|
	\int_{2^{-n+l-1} \le |\tx-y|\le 2^{-n+l}}
	\frac{2^{jd} }{(1+2^j|\tx-y|)^{d+\eps}} |h(y)|dy
	\\
	&\lc 2^{(j-n)(1-\eps)} 2^{-l\eps} \intslash_{B(\tx, 2^{l-n})} |h(y)| dy.
	\end{align*}
	Summing in $\l>0$ and then $j\le n$ yields 
	\Be
	\sum_{j\le n} 
	|K_j*h(\tx)-K_j*h(w)| \lc Mh(\tx) \lc Mf(\tx). \qedhere
	\Ee
\end{proof} 

\begin{proof}[Proof of \eqref{Sstarineq}]
	We proceed arguing as in \cite[\S I.7]{Ste93}.
	Fix $\tx\in \bbR^d$ and $n\in \bbZ$ and define $g$ and $h$ as in the lemma. For (suitable) $w$ with $|w-\tx|\le 2^{-n-1}$ we write
	\begin{align}
	S_n f(\tx)&= S_n g(\tx)+ (S_n-S)h(\tx) + Sh(\tx)
	\notag \\
	\label{cotlardecomp}
	&=S_n g(\tx)+ (S_n-S)h(\tx) + Sh(\tx)- Sh(w)
	+ Sf(w)- Sg(w).
	\end{align}
	By Lemma \ref{cotlarlemmaapp}
	\[|S_n g(\tx)|+ |(S_n-S)h(\tx)| + |Sh(\tx)- Sh(w)|
	\lc  B \,M[f](\tx)\]
	and it  remains to consider the term $Sf(w)-Sg(w)$ for a substantial set  of $w$ with $|w-\tx|\le 2^{-n-1}$.  
	
	By the Mikhlin-H\"ormander theorem we have for all $f\in L^1(\bbR^d)$ and all $\lambda>0$
	\[\meas( \{x: |Sf(x)|>\la\}) \le A \la^{-1} \|f\|_1 \]
	where $A \le C_{\alpha,d}B$.
	
	Now let $\delta\in (0,1/2)$ and consider the 
	set
	$$\Omega_n(\tx, \delta)=\big\{ w: |w-\tx|< 2^{-n-1},
	\quad |Sg(w)|> 2^d\delta^{-1} A \,M[f](\tx)\big\}.$$
	In  \eqref{cotlardecomp} we  can estimate
	the term $|Sg(w)|$ by
	$2^d\delta^{-1} A \,M[f](\tx)$ when
	$w\in B(\tx, 2^{-n-1}) \setminus \Om_n(\tx,\delta)$.
	Hence we obtain
	\Be\label{infineq}
	|S_n f(\tx)| \le 
	\inf_{w\in B(\tx, 2^{-n-1}) \setminus \Om_n(\tx,\delta)} 
	|Sf(w)| 
	+ C(\alpha, d) B (1+\delta^{-1}) M[f](\tx).
	\Ee By the weak type inequality for $S$ we have
	\begin{align*} \meas (\Omega_n(\tx,\delta)) &\le 
	\frac{A\|g\|_1}{2^d\delta^{-1} A M[f](\tx)}
	=\frac{\delta}{2^d M[f](\tx)}\int_{|\tx-y|\le 2^{-n}}|f(y)| dy
	\\
	&\le \delta \, 2^{-d}\,\meas (B(\tx, 2^{-n}))
	=\delta \, \meas (B(\tx, 2^{-n-1})).
	\end{align*}
	Hence  $\meas( B(\tx, 2^{-n-1}) \setminus \Om_n(\tx,\delta))\ge (1-\delta)  \meas (B(\tx, 2^{-n-1}))$
	and thus for all $r>0$ 
	\begin{align*}
	&\inf_{w\in B(\tx, 2^{-n-1}) \setminus \Om_n(\tx,\delta)} 
	|Sf(w)| \\ &\le \Big(\frac{1} {\meas(B(\tx, 2^{-n-1}) \setminus \Om_n(\tx,\delta))} \int_{B(\tx, 2^{-n-1})}  |Sf (w)|^r dw\Big)^{1/r}
	\\
	&\le \Big(\frac{1} {(1-\delta)|B(\tx, 2^{-n-1})|}  \int_{B(\tx, 2^{-n-1})}  |Sf (w)|^r dw\Big)^{1/r}.
	\end{align*} We obtain 
	\[
	|S_n f(\tx)| \le  (1-\delta)^{-1/r} (M[|Sw|^r](\tx))^{1/r}
	+ C(\alpha,d) (1+\delta^{-1}) B\,M[f](\tx)\]
	uniformly in $n$. This implies \eqref{Sstarineq}.
\end{proof}

\section{ Proof of the  Chang-Wilson-Wolff inequality}\label{appendixB}

In this section we prove Proposition \ref{prop:CWW85}. For $m\in\Z$ we define
\begin{align*} \mathcal{M}_m f(x) &= \sup_{j\ge m} |\mathbb{E}_j f(x) - \mathbb{E}_m f(x) |, 
\\
\mathfrak{S}_m f(x) &= \Big( \sum_{j=m}^\infty |\mathbb{D}_j f(x)|^2\Big)^{1/2}.  
\end{align*} 
We show that for real valued $f\in L^\infty(\bbR^d)$,

\Be\label{CWWineq0}\begin{aligned} \meas\Big(
	&\Big \{x \in \R^d \colon |f(x)-\mathbb{E}_0 f(x)| > 2 \lambda \text{ and } \fS_0 f(x) \leq \varepsilon \lambda \Big \} \Big)
	\\
	&\qquad
	\leq 2 \exp(- \frac{(1-\eps)^2}{2\eps^2}) \meas\Big(\Big\{x \in \R^d \colon \mathcal{M}_0 f(x) > \lambda \Big \} \Big).
\end{aligned}\Ee
We shall give the proof for the convenience of the reader. It is due to Herman Rubin (simplifying an earlier argument by Chang, Wilson and Wolff as explained in \cite{CWW85}).

First, we claim that if $n$ is a non-negative integer, $I_n$ is a dyadic cube of side length $2^{-n}$ in $\R^d$, and $I_{n,a} := \{x \in I_n \colon \fS_0 f(x) < a \}$ where $a > 0$, then\begin{subequations} 
\label{eq:|int_e_f-E_0|}
\begin{align} 
\label{eq:int_e_f-E_0}
&\frac{1}{|I_n|} \int_{I_{n,a}} e^{t [f(x)-\bbE_n f(x)]} dx \leq e^{\frac{1}{2} t^2 a^2},
\\
\label{eq:-(int_e_f-E_0)}
&\frac{1}{|I_n|} \int_{I_{n,a}} e^{-t [f(x)-\bbE_n f(x)]} dx \leq e^{\frac{1}{2} t^2 a^2}.
\end{align}
\end{subequations}
for every $t > 0$. Indeed, for every such $I_n$, $a$ and $t$, we have, by the Lebesgue differentiation theorem,
and dominated convergence,  that
$$
\frac{1}{|I_n|} \int_{I_{n,a}} e^{t [f(x)-\bbE_n f(x)]} dx
= \lim_{m \to \infty} \frac{1}{|I_n|}  \int_{I_{n,a}} e^{t [\bbE_m f(x) - \bbE_n f(x)]} dx,
$$
while for every $m \geq n$, 
\begin{multline} \label{eq:proof_martingale}
\frac{1}{|I_n|} \int_{I_{n,a}} e^{t [\bbE_m f(x) - \bbE_n f(x)]} dx\\
\leq \frac{1}{|I_n|} \int_{I_n} \frac{ e^{t [E_m f(x) - E_n f(x)]} }{ \prod_{j=n}^{m-1} \bbE_j(e^{t \bbD_j f})(x) } dx \,\,  \Big \| \prod_{j=n}^{m-1} E_j(e^{t \bbD_j f}) \Big \|_{L^{\infty}(I_{n,a})} .
\end{multline}
But 
\beq \label{eq:martingale}
\frac{1}{|I_n|} \int_{I_n} \frac{ e^{t [\bbE_m f(x) - \bbE_n f(x)]} }{ \prod_{j=n}^{m-1} \bbE_j(e^{t \bbD_j f})(x) } dx = 1
\endeq
for every $m \geq n$, since the integrand forms a martingale on $I_n$. More precisely, (\ref{eq:martingale}) is clearly true if $m = n$, and if this is true for some $m \geq n$, then for any dyadic cube $I_m$ of side length $2^{-m}$ inside $I_n$, we have
\[
\int_{I_m} \frac{ e^{t [\bbE_{m+1} f(x) - \bbE_n f(x)]} }{ \prod_{j=n}^{m} \bbE_j(e^{t \bbD_j f})(x) } dx 
= \int_{I_m} \frac{ e^{t \bbD_m f(x)}}{ \bbE_m(e^{t \bbD_m f})(x) } \cdot \frac{ e^{t [\bbE_m f(x) - \bbE_n f(x)]} }{\prod_{j=n}^{m-1} \bbE_j(e^{t \bbD_j f})(x) } dx
\]
Since the second fraction in the integrand is constant on $I_m$, this gives
\[
\int_{I_m} \frac{ e^{t [\bbE_{m+1} f(x) - \bbE_n f(x)]} }{ \prod_{j=n}^{m} \bbE_j(e^{t \bbD_j f})(x) } dx
= \int_{I_m} \frac{ e^{t [\bbE_m f(x) - \bbE_n f(x)]} }{ \prod_{j=n}^{m-1} \bbE_j(e^{t \bbD_j f})(x) } dx,
\]
which gives (\ref{eq:martingale}) for $m+1$ in place of $m$ upon summing over all the $I_m$'s inside $I_n$ and using the induction hypothesis. Now from $\mathbb{E}_j(\bbD_j f) = 0$, we have
$$
\bbE_j(e^{t \bbD_j f})(x) = \cosh( t \bbD_j f(x) ) \leq e^{\frac{1}{2} t^2 |\bbD_j f(x)|^2}
$$
for all $x$, so
$$
\prod_{j=n}^{m-1} \bbE_j(e^{t \bbD_j f})(x) \leq e^{\frac{1}{2} t^2 \fS_0 f(x)^2},
$$
which gives
$$
\Big \| \prod_{j=n}^{m-1} \bbE_j(e^{t \bbD_j f}) \Big \|_{L^{\infty}(I_{n,a})}
\leq e^{\frac{1}{2} t^2 a^2}.
$$
In view of (\ref{eq:proof_martingale}) and (\ref{eq:martingale}), we have established our claim (\ref{eq:int_e_f-E_0}).
 Replacing $f$ by $-f$ we also obtain 
\eqref{eq:-(int_e_f-E_0)}.

Now consider any $n \geq 0$ and any dyadic cube  $I_n$ of side length $2^{-n}$. From \eqref{eq:int_e_f-E_0}, \eqref{eq:-(int_e_f-E_0)} and Chebyshev's inequality, we have, for any $\lambda > 0$ and $a > 0$, that
$$
\meas\big(\{x \in I_n \colon |f(x) - \bbE_n f(x)| > \lambda \text{ and } \fS_0 f(x) < a\}\big) \leq 2e^{-t \lambda} e^{\frac{1}{2} t^2 a^2} |I_n|
$$
for all $t > 0$, so minimizing over $t > 0$ (i.e. setting $t = \lambda a^{-2}$), we have
\beq \label{eq:f-EnChebeyshev}
\meas\big(\{x \in I_n \colon |f(x) - \bbE_n f(x)| > \lambda \text{ and } \fS_0 f(x) < a \}\big) \leq 2 e^{-\frac{\lambda^2}{2 a^2}} |I_n|
\endeq

Let $I_0$ be any dyadic cube of side length $1$, and let $\mathcal{I}$ be a collection of maximal dyadic subcubes $I$ of $I_0$ such that 
$$
\left | \frac{1}{|I|} \int_I (f - \bbE_0 f) \right | > \lambda.
$$ 
For each $I \in \mathcal{I}$ consider the following subset of $I$:
$$
\{ x \in I \colon |f(x) - \bbE_0 f(x)| > 2 \lambda \text{ and } \fS_0 f(x) \leq \varepsilon \lambda \}.
$$
If this subset of $I$ is non-empty and $|I| = 2^{-n}$, then by considering the dyadic parent of $I$ and using the existence of $x \in I$ where $\fS_0 f(x) \leq \varepsilon \lambda$, in particular  $|\bbE_{n-1} f- \bbE_n f|\le \eps\la$,  we have that 
$$
\Big| \frac{1}{|I|} \int_I (f - \bbE_0 f) \Big| \leq (1 + \varepsilon) \lambda,
$$
and so
$$
|\bbE_n (f- \bbE_0 f)(x)| \leq (1 + \varepsilon) \lambda
$$
for every $x \in I$. It follows that
\[
\begin{split}
&\{ x \in I \colon |f(x) - \bbE_0 f(x)| > 2 \lambda \text{ and } \fS_0 f(x) \leq \varepsilon \lambda \} \\
&\subseteq 
\{ x \in I \colon |(f-\bbE_0 f)(x) -\bbE_n(f-\bbE_0 f)(x)| > (1-\varepsilon) \lambda \text{ and } \fS_0 f(x) \leq \varepsilon \lambda \},
\end{split}
\]
which by (\ref{eq:f-EnChebeyshev}) (applied to $f-\bbE_0 f$ instead of $f$) has measure bounded by $2 e^{-\frac{(1-\varepsilon)^2}{2 \varepsilon^2}} |I|$. 
Since this is true for all $I \in \mathcal{I}$, summing over all $I \in \mathcal{I}$, we get 
\begin{align*}
\meas\big( \{x \in I_0 \colon |f(x) - \bbE_0 f(x)| > 2 \lambda 
\text{ and } \fS_0 f(x) \leq \varepsilon \lambda \} \big)&
\\
\leq  2e^{-\frac{(1-\eps)^2}{2 \varepsilon^{2}}}
 \meas\big( \{x \in I_0 \colon \mathcal{M}_0 f(x) > \lambda \} \big)&
\end{align*}
for all $\lambda > 0$ and $0 < \varepsilon < 1$. Summing over all dyadic cubes $I_0$ of side length $1$, we get the desired conclusion in \eqref{CWWineq0}.

To prove 
\eqref{CWWineq} for real-valued functions  we use a scaling argument,  applying the above to  $f(2^N \cdot)$.
This leads to
\begin{multline} \label{eqn:cww2}
 \meas(\{ x\in\R^d\,:\, |f(x)-\mathbb{E}_{-N} f(x)|> 2\lambda,\,\mathfrak{S}_{-N} f(x)\le \varepsilon \lambda \}) \\
\le  2\exp\big(- \tfrac{(1-\eps)^2}{2\eps^2}\big) 
\meas\big(\{ x\in\R^d\,:\, \mathcal{M}_{-N} f(x)>\lambda\}\big). \end{multline}
Since $f\in L^p(\R^d)$ we have $\|\bbE_{-N} f\|_\infty \le 2^{-Nd/p}\|f\|_p$ and thus
 $\mathbb{E}_{-N} f\to 0$ uniformly  as $N\to\infty$. 
 Let $0<\delta\ll 1$. Pick $N$ such that 
 \[2^{-Nd/p}\|f\|_p<\delta.\]
  Then $|f(x)|>2\la+\delta$ implies 
 $|f(x)-\bbE_{-N} f(x)|> 2\la$ for such a choice of $N$. We also have $\fS_{-N} f(x)\le \fS f(x)$, $\cM_{-N} f(x)\le \cM f(x)$,  and thus, for $\eps<1$,
 \begin{align*}
 & \meas(\{ x\in\R^d\,:\, |f(x)|>2\la+\delta, \,\mathfrak S f(x)\le \eps\la\})
 \\
 &\quad
 \le \meas(\{ x\in\R^d\,:\, |f(x)-\mathbb{E}_{-N} f(x)|> 2\lambda,\,\mathfrak{S}_{-N} f(x)\le \varepsilon \lambda \}) \\
&\quad \le 2 \exp\big(- \tfrac{(1-\eps)^2}{2\eps^2}\big) 
\meas\big(\{ x\in\R^d\,:\, \mathcal{M}_{-N} f(x) > \lambda \}) 
\\
 &\quad\le  2\exp\big(- \tfrac{(1-\eps)^2}{2\eps^2}\big) 
\meas\big(\{ x\in\R^d\,:\, \mathcal{M} f(x) > \lambda\}) .
 \end{align*}
 We let $\delta\to 0$  and obtain
 \begin{multline*}\label{globalcwwforrealvalued}
  \meas(\{ x\in\R^d\,:\, |f(x)|>2\la, \mathfrak S f(x)\le \eps\la\})
 \\ \le 2 \exp\big(- \tfrac{(1-\eps)^2}{2\eps^2}\big) 
\meas\big(\{ x\in\R^d\,:\, \mathcal{M} f(x) > \lambda\}\big).
 \end{multline*}
 
For complex valued functions we apply \eqref{CWWineq0}  to the  real and imaginary parts and we obtain
\Be\label{globalcwwforcplvalued} 
\begin{aligned}
 & \meas(\{ x\in\R^d\,:\, |f(x)|>2\sqrt 2\la, \mathfrak S f(x)\le \eps\la\})
 \\& \le 
 4 \exp\big(- \tfrac{(1-\eps)^2}{2\eps^2}\big) 
\meas\big(\{ x\in\R^d\,:\, \mathcal{M} f(x) > \lambda\}\big).
 \end{aligned}
\Ee
In particular we obtain Proposition \ref{prop:CWW85} 
(where $\eps<1/2$)
with the constants $c_1=1/8$ and $c_2=4$. \qed
\newcommand{\etalchar}[1]{$^{#1}$}

\newpage

\end{document}